\documentclass[a4paper,12pt]{article}

\usepackage{mystyle}

\title{Congruence frames of frames and $\kappa$-frames}
\author{Graham Manuell}
\date{August 2015}

\begin{document}

\pagenumbering{Alph}
\begin{titlepage}
  \begin{center}
    \vspace*{1cm}
    
    \Huge
    \textbf{\MyTitle}
    
    \vspace{1.0cm}
    
    \Large
    \MyAuthor
    
    \vfill
    
    \large
    A thesis presented for the degree of Master of Science
    in the Department of Mathematics and Applied Mathematics, University of Cape Town
    
    \vspace{0.2cm}
    
    Supervisors: Dr J.\ Frith and Professor C.R.A.\ Gilmour
    
    \vspace{0.3cm}
    
    \MyDate
  \end{center}
\end{titlepage}

\subsection*{Acknowledgements}
\pagenumbering{roman}

I would like to thank my supervisors Dr J.\ Frith and Professor C.R.A.\ Gilmour for their help, encouragement and patience.
I am grateful for the freedom to pursue a topic that I found interesting.

I acknowledge financial support from the National Research Foundation and thank the department of Mathematics
and Applied Mathematics for the use of their facilities.

I am also grateful to my friends for their companionship.

Finally, I would like to thank my parents for their support. Without their assistance my studies would not have been possible.

\tableofcontents

\setcounter{section}{-1}

\section{Introduction}\label{section:introduction}
\pagenumbering{arabic}

A notable feature of frame theory is that the lattice of congruences on a frame is a frame itself \cite{QuotientFrames}.
This frame of congruences is usefully employed in studying the topological aspects of frames and reappears in many pointfree constructions.
The result also holds in the more general setting of congruences on $\kappa$-frames \cite{Madden}.

The congruence frame is studied in different guises and is called different names. Those who prefer to work with locales call it the \emph{dissolution locale},
but generally prefer to work with its dual, the lattice of sublocales. Locale theorists often make heavy use of this lattice and place it on equal footing
with the locale's frame of opens. The earlier sections of this thesis use the congruence frame in a similar way, reformulating the locale-theoretic notions
frame-theoretically so that they are more generally applicable.

Frame theorists with a more order-theoretic viewpoint often work with nuclei and call the frame of nuclei the \emph{assembly}
of a frame. We prefer the more algebraic approach using congruences. Congruences are a natural tool for dealing with the
quotients of an algebraic structure such as frames, so it is somewhat surprising that they are not used more. Importantly for us,
they also allow for generalisation to $\kappa$-frames with little modification. The first thorough description of the assembly in
terms of congruences instead of nuclei was by Frith in \cite{FrithCong}. The proofs of the basic results are not only more general,
but arguably cleaner with this approach.

\subsection*{Outline}

\subsubsection*{Section 1}

\Cref{section:frames} provides the necessary background on frames, $\kappa$-frames and biframes. No new results appear in this section.

\subsubsection*{Section 2}

\Cref{section:congruences} describes the notion of a congruence and how it applies to frames. We recount the
relationships between congruences on frames, nuclei and subspaces before moving on to open and closed congruences in \cref{subsec:closed_open_congs} where
we closely follow the presentation in \cite{FrithCong}, noting that the proofs hold for general $\kappa$-frames.
Here we also introduce for the first time the notion of a \emph{generalised closed} congruence on a $\kappa$-frame as well as that of the generalised
closure and describe how closed congruences and closure behave in quotient frames. We make heavy use of the generalised closure throughout this section.
\Cref{lem:closed_quotient_of_quotient} is the analogue of how closed sets restrict to subspaces in topology and is no doubt known for frames,
but I have not found a proof of it in the literature.

In \cref{subsec:dense_congruences} we discuss the notion of density and the largest dense congruence on a frame. These concepts are well-understood
for frames and Madden discusses the situation for $\kappa$-frames in \cite{Madden}. However, Madden makes no mention of the largest congruence dense in
a given congruence and these important congruences have consequently escaped naming until now. In the case of frames, they might be called Boolean
congruences since their corresponding quotients are Boolean frames, but this is not the case for $\kappa$-frames. We call them \emph{clear} congruences,
since we might imagine the smallest dense sublocale to be so rarefied as to appear translucent. Exploring the concept of clear congruences leads to a
number of original results. \Cref{lem:meet_of_clear} states that every congruence is a meet of clear congruences. This was known for frames, but the
result for $\kappa$-frames and our approach to the proof are new. What Madden calls \emph{d-reduced} $\kappa$-frames in \cite{Madden}, we call \emph{clear}
$\kappa$-frames and we show in \cref{lem:quotient_by_clear} that the clear quotients of a $\kappa$-frame are precisely its quotients by clear congruences.
Madden claims what amounts to one direction of this result, but the converse can only be articulated once the notion of a clear congruence is established.
\Cref{cor:clear_frame_Boolean}, that (in our terminology) the clear frames and the Boolean frames coincide, is well-known, but it is included for completeness
and used to give a new characterisation of clear $\kappa$-frames in \cref{thm:clear_kappa_frames}. The characterisation of Boolean $\kappa$-frames as
hereditarily clear is also new.

In \cref{subsec:congruence_kappa_frames} we recount the proof that the congruence lattices of $\kappa$-frames are frames. We follow \cite{Madden} in
constructing the congruence frame of a $\kappa$-frame by freely adjoining complements to the $\kappa$-frame and then considering $\kappa$-ideals on the result.
In this subsection, \cref{prop:congruence_kappa_frame_summary} appears to contain a new description of the congruence $\kappa$-frame of a quotient frame.
\Cref{lem:kernel_on_congruence_frame} gives the form of the kernel of any map from a congruence $\kappa$-frame. I haven't seen this proved anywhere else,
so it might also be new, though I would be surprised if its more specific corollary, at least, were not already known.

\Cref{subsec:congruence_tower} is a brief discussion of known results about the congruence tower. The results in \cref{subsec:clear_congruences_on_frames}
concern clear congruences on frames and are also known, but we believe the proofs to be new. In \cite{PleweRareSublocales} Plewe introduces rare sublocales
and uses them to prove results about when congruence frames are Boolean. \Cref{subsec:boolean_congruence_frames} translates the most basic of these in
terms of congruences. Many of the proofs are different to those seen before, but the only new results are the description of rare congruences in
\cref{lem:rare_congruence} and its application to a classification of complete chains with Boolean congruence frame in \cref{lem:no_dense_elements_Boolean},
the one direction of which was known to Beazer and Macnab.

\subsubsection*{Section 3}

\Cref{section:strictly_zero_dimensional_biframes} introduces the \emph{category} of strictly zero-dimensional biframes. Strictly zero-dimensional biframes
were first defined in \cite{StrictlyZeroDimensional}, but the category hasn't been considered before. The biframe structure on the congruence frame was
first described by Frith in \cite{FrithCong}. For this section, we restrict our attention to congruence frames of frames, but many of the results should
easily generalise to congruence $\kappa$-frames. Almost all of the results are new. In \cref{prop:congruence_free_str0dbifrm}, the congruence functor is
shown to be left adjoint to the functor which takes the first part of a strictly zero-dimensional biframe. While this result is simply a rephrasing of
the usual universal property of the congruence frame, we believe the new description as an adjoint to be rather elegant.

We propose that strictly zero-dimensional biframes be viewed as generalised frames in a similar way to how the locale theorists view locales.
The first part of a strictly zero-dimensional biframe gives what we may think of as the frame of opens, while the total part gives what we may think of
as the designated frame of congruences. This view is justified in \cref{subsec:str0d_biframes_of_congruences}.

\Cref{thm:dense_quotients_of_congruence_frame} in \cref{subsec:characterisations_of_str0d_biframes} gives a new characterisation of strictly zero-dimensional
biframes as dense quotients of congruence biframes. The rest of the subsection explores some of the consequences of this result.

We discuss compact strictly zero-dimensional biframes in \cref{subsec:cpt_str0d_biframes}. All such biframes are congruence biframes and this gives
an equivalence between the category of compact strictly zero-dimensional biframes and the category of Noetherian frames. In contrast,
\cref{ex:lindelof_str0d_biframe} describes a Lindelöf strictly zero-dimensional biframe which is not a congruence biframe. We also take a short
digression to correct a mistake in \cite{BiframeThesis} by showing in \cref{thm:compact_Boolean_biframe_finite} that a compact Boolean biframe is finite.
I do not believe this has been proven in the literature despite following from the known result that a distributive lattice is finite if it satisfies both
the ascending and descending chain conditions \cite[Theorem~4.28]{LatticesAndOrderedSets}. We are careful to only assume dependent choice in
our proof.

We discuss Skula biframes in \cref{subsec:skula_biframe}. These were also discussed in the original paper on strictly zero-dimensional
biframes \cite{StrictlyZeroDimensional}, but we consider them in the framework developed in the previous sections. The Skula functor
gives a dual embedding of the category of $T_0$ spaces into the category of strictly zero-dimensional biframes so that even some results about
non-sober spaces can be given a pointfree interpretation. The main result in this subsection is \cref{lem:skula_spatial_reflection_of_congruence}, which
claims that the Skula biframe on the spectrum of a frame $L$ is naturally isomorphic to the spatial reflection of the congruence frame of $L$.
This isomorphism is well-known, though our proof is new. It would be surprising if the naturality of this isomorphism were not known, but I could not find a
proof of it in the literature. All the ingredients needed to show \cref{lem:skula_versus_smallest_str0d_biframe} are known and unremarkable, but this is
the first time they have been expressed in this way.

\Cref{subsec:which_morphisms_lift} examines the question of which morphisms of frames are given by the first part of
a morphism between discrete strictly zero-dimensional biframes. Preliminary results are established, but we hope that the classification
in the injective case can be improved.

Finally, in \cref{subsec:clear_elements} we introduce the notion of a \emph{clear element} of a strictly zero-dimensional biframe and show how these provide
a satisfactory generalisation of the clear congruences of a congruence frame. We then use them in
\cref{lem:str0d_biframe_congruential_iff_no_missing_clear_elements} to provide a characterisation of congruence biframes as strictly zero-dimensional
biframes with no `missing' clear elements.

\subsubsection*{Section 4}

\Cref{section:reflections_of_congruence_frames} describes the interaction of congruence frames with the spatial reflection, universal biframe compactification
and the construction of the free frame on a $\kappa$-frame. The main results are new.

The results in \cref{subsec:spatial_reflection} about prime congruences and the spatial reflection of the congruence frame are known in the case of frames,
but our approach is new. To generalise the results to $\kappa$-frames, we introduce the notion of a $\kappa$-space, which is a natural generalisation of the
$\sigma$-spaces of \cite{SigmaSpaces}.

In \cref{subsec:congruence_frame_compactification} we discuss the relationship between the frame congruences on a frame and the lattice
of $\kappa$-frame congruences on the same frame. The main result is that the frame of lattice congruences is the universal biframe compactification
of the congruence frame. The question of whether this result generalises to the frame of $\sigma$-frame congruences on a frame is addressed, but
ultimately left unsolved. The characterisation of hereditarily $\kappa$-Lindelöf frames is surely well-known, but all the other results are original.

Most of the results in \cref{subsec:lindelof_congruences} appear to be original. Motivated by the equivalence between the categories of $\kappa$-frames
and $\kappa$-coherent frames, we establish a link between the congruence frame of a $\kappa$-frame and the congruence frame of the free frame generated
by the $\kappa$-frame in \cref{lem:congruences_on_frames_of_ideals}. The $\kappa = \aleph_0$ was dealt with in \cite{CoherentCongruences}, but our
approach is different. We then apply this result to describe the Lindelöfication of certain quotients of a completely regular Lindelöf frame.
This provides a pointfree analogue of a result from \cite{ExtensionsOfZeroSets} that the realcompactification of a z-embedded subspace of a realcompact
space $X$ is given by its $G_\delta$-closure in $X$.

\section{Background}\label{section:frames}

We briefly give the definitions and results we will use later on. For further background on frames and $\kappa$-frames see
\cite{PicadoPultr} and \cite{Madden}. Biframes were introduced in \cite{Biframes}.

\subsection{Frames and \texorpdfstring{$\kappa$}{κ}-frames}

A frame is a complete lattice satisfying the frame distributivity condition
\[x \wedge \bigvee_{\alpha \in I} x_\alpha = \bigvee_{\alpha \in I} (x \wedge x_\alpha)\]
for arbitrary families $(x_\alpha)_{\alpha \in I}$. We denote the smallest element of a frame
by $0$ and the largest element by $1$. A frame homomorphism is a function which preserves finite
meets and arbitrary joins (and thus in particular $0$ and $1$).

An infinite cardinal number $\kappa$ is called \emph{regular} if any cardinal sum $\sum_{\alpha \in I} \lambda_\alpha < \kappa$ whenever
$\lambda_\alpha < \kappa$ for all $\alpha \in I$ and $|I| < \kappa$. In this thesis $\kappa$ will always denote a regular cardinal without comment.
A $\kappa$-set is then a set of cardinality strictly less than $\kappa$.
A $\kappa$-frame is a bounded distributive lattice which has joins of $\kappa$-sets and satisfies the frame distributivity law
when $|I| < \kappa$. Morphisms of $\kappa$-frames are functions which preserve finite meets and $\kappa$-joins (i.e.\ joins of $\kappa$-sets).

The category of frames is denoted by $\Frm$ and the category of $\kappa$-frames by $\kappa\Frm$.
\emph{Throughout this thesis when we prove results for $\kappa$-frames, they will also hold for frames
(interpreting $\kappa$-joins as arbitrary joins where necessary) unless we explicitly say otherwise.}

In $\mathsf{ZF}$ set theory the only provably regular cardinal is $\aleph_0$ and so the only examples of $\kappa$-frames are bounded distributive lattices.
The regularity of $\aleph_1$ follows from countable choice and all infinite successor cardinals are regular in $\mathsf{ZFC}$.
The case of $\aleph_1$-frames is well studied and these are better known as $\sigma$-frames \cite{Charalambous}.
When dealing with $\kappa$-frames we will silently assume as much choice as necessary. In particular, results concerning $\sigma$-frames will
often assume countable choice. When working with frames we will be more careful.

\subsection{Regular and completely regular \texorpdfstring{$\kappa$}{κ}-frames}

If $L$ is a $\kappa$-frame (or frame) and $a, b \in L$, we write $a \prec b$ and say that $a$ is \emph{rather below} $b$ when there
is a \emph{separating element} $z \in L$ satisfying $a \wedge z = 0$ and $b \vee z = 1$.
A $\kappa$-frame $L$ is \emph{regular} if every element is expressible as a $\kappa$-join of elements rather below it.

Similarly, we write $a \prec\prec b$ and say that $a$ is \emph{completely below} $b$ in $L$ if there is a family $\{z_q \in L \mid q \in [0,1] \cap \Q\}$
such that $z_0 = a$, $z_1 = b$ and $z_r \prec z_s$ whenever $r < s$. We say that $\prec\prec$ interpolates since whenever $a \prec\prec b$
there exists a $c \in L$ such that $a \prec\prec c \prec\prec b$. Assuming dependent choice, the relation $\prec\prec$ is the largest
interpolating relation contained in $\prec$. We say that $L$ is \emph{completely regular} if every element is expressible as a $\kappa$-join
of elements completely below it.

An element $a$ in a $\kappa$-frame $L$ is said to be \emph{complemented} if there is a $b \in L$ such that $a \wedge b = 0$ and $a \vee b = 1$.
The element $b$ is uniquely determined and is said to be the \emph{complement} of $a$. We write the complement of $a$ as $a^c$. Note that the set of
complemented elements of a $\kappa$-frame is closed under finite meets and finite joins.
An element $a \in L$ is complemented if and only if $a \prec a$ and we call $L$ \emph{zero-dimensional} if every element is a $\kappa$-join of
complemented elements.

The monomorphisms in the category of regular $\kappa$-frames are the \emph{dense} $\kappa$-frame homomorphisms. These are the maps $f$ such that
$f(x) = 0$ only when $x = 0$. On the other hand, a $\kappa$-frame homomorphism $f$ for which $f(x) = 1$ only when $x = 1$ is called \emph{codense}.
Codense maps from regular $\kappa$-frames are always injective.

\subsection{Heyting algebras and right adjoints}

Let $P$ and $Q$ be posets. A \emph{Galois connection} between $P$ and $Q$ consists of monotone maps $f: P \to Q$ and $g: Q \to P$ such that
$f(x) \le y \iff x \le g(y)$ for all $x \in P$ and all $y \in Q$. We say that $f$ is the left adjoint of $g$ and that $g$ is the right adjoint of $f$.

A left adjoint $f: P \to Q$ is surjective if and only if its right adjoint $g$ is injective if and only if $fg = 1$ and
$f$ is injective if and only if $g$ is surjective if and only if $gf = 1$.

A monotone map $f: P \to Q$ between complete lattices has a right adjoint if and only if it preserves joins and in this case its right adjoint,
denoted by $f_*$, is uniquely defined by $f_*(y) = \bigvee\{x \in Q \mid f(x) \le y\}$. In particular, frame homomorphisms always have right adjoints.

Let $L$ be a frame. The frame distributivity law implies that the map $x \mapsto x \wedge a$ preserves arbitrary joins
and thus has a right adjoint $x \mapsto (a \rightarrow x)$ turning $L$ into a complete Heyting algebra.

If $a \in L$, the element $a \rightarrow 0$ is called the pseudocomplement of $a$ and is denoted by $a^*$.
It is the largest element $c \in L$ for which $a \wedge c = 0$. Notice that if $a$ is complemented, its complement
is given by $a^*$, and that $a \prec b$ if and only if $a^* \vee b = 1$.
The pseudocomplement satisfies $a \le a^{**}$, $a^{***} = a^*$, $(a \vee b)^* = a^* \wedge b^*$ and $(a \wedge b)^{**} = a^{**} \wedge b^{**}$.

\subsection{The adjunction between frames and spaces}

Every topological space $(X,\tau)$ has an associated frame of open sets, $\tau$.
This can be extended to a functor $\Omega: \Top\op \to \Frm$ by defining $(\Omega f)(U) = f^{-1}(U)$.

Starting with a frame, we may construct a topological space known as its \emph{spectrum}.
An element $p$ of a frame $L$ is said to be \emph{prime} if $p \ne 1$ and $a \wedge b \le p$ implies $a \le p$ or $b \le p$.
The set of prime elements of $L$ may be endowed with a topology given by open sets of the form $U_a = \{p \text{ prime} \mid a \not\le p\}$
for each $a \in L$. It can be shown that the right adjoints of frame homomorphisms preserve prime elements and we obtain a functor
$\Sigma: \Frm\op \to \Top$ with $(\Sigma f)(p) = f_*(p)$.

These give a contravariant adjunction between the category of frames and the category of topological spaces
with the functor $\Sigma\op$ left adjoint to $\Omega$. The unit $\sigma: 1_{\Frm} \to \Omega\Sigma$ and
the counit $\sob: 1_{\Top} \to \Sigma\Omega$ of this adjunction are given by
$\sigma_L(a) = U_a$ and  $\sob_X(x) = \cl(x)^c$ where $\cl(x)$ is the closure of $\{x\}$ in $X$ and $\cl(x)^c = X \setminus \cl(x)$.

The frame $\Omega\Sigma L$ together with the map $\sigma_L$ is known as the \emph{spatial reflection} of $L$.
We say that $L$ is spatial if and only if $\sigma_L$ is an isomorphism. Notice that any subframe of a spatial frame is spatial.
A frame $L$ is spatial if and only if every element of $L$ is a meet of prime elements.
In this way, spatiality can be viewed as an algebraic property akin to prime factorisation in rings.
The map $\sigma_L$ is surjective and $\Spat L$ may be thought
of as the largest spatial quotient of $L$. As in \cref{subsec:frame_congruences}, we might represent $\sigma_L$
as a nucleus $x \mapsto \bigwedge\{p \in L \mid p \ge x, p \text{ prime}\}$. We call the elements which are fixed by
this nucleus \emph{spatial elements} of $L$.

The space $\Sigma\Omega X$ together with the map $\sob_X$ is known as the \emph{sobrification} of $X$.
We say $X$ is sober if $\sob_X$ is an isomorphism. This holds if and only if every irreducible closed set in $X$
is the closure of a unique point and is implied by the Hausdorff property. The map $\sob_X$ is an injection if and only if $X$ is $T_0$.

\subsection{Frames of ideals}

A subset $D$ of a poset is called \emph{$\kappa$-directed} if every $\kappa$-set $S \subseteq D$ has an upper bound in $D$.
If $X$ is a set, the set of all subsets of $X$ of cardinality less than $\kappa$ is $\kappa$-directed.
Thus, an arbitrary join can be decomposed as a $\kappa$-directed join of $\kappa$-joins. An $\aleph_0$-directed set is said to be
\emph{directed}.

A $\kappa$-ideal on a $\kappa$-frame $L$ is a $\kappa$-directed downset on $L$ (i.e.\ a downset which is closed under $\kappa$-joins).
The set of $\kappa$-ideals on a $\kappa$-frame ordered by inclusion forms a frame. There is an obvious forgetful functor from the
category of frames to the category of $\kappa$-frames. This has a left adjoint $\h_\kappa: \kappa\Frm \to \Frm$ where $\h_\kappa L$
is the frame of $\kappa$-ideals on $L$ and $(\h_\kappa f)(I) = \,\downarrow\!\!f(I)$. The unit map sends each element of $L$ to the
principal ideal generated by it, while the counit maps a $\kappa$-ideal to its join. Note that the counit map is a dense frame homomorphism.

The frame of all ideals is given by $\h_{\aleph_0} L = \ideal L$ and the frame of $\sigma$-ideals
is given by $\h_{\aleph_1} L = \h L$. In the interest uniform notation, we will use $\h_\infty L$ to denote the frame of principal ideals on
a frame $L$. This is isomorphic to the frame $L$ itself.

Similarly, there is a forgetful functor from $\Frm$ to the category $\Slat$ of meet-semilattices, which has the downset functor
$\mathscr{D}: \Slat \to \Frm$ as its left adjoint. Here $\mathscr{D} L$ is the frame of downsets on the meet-semilattice $L$ and
$(\mathscr{D} f)(D) = \,\downarrow\!f(D)$.

It is well-known that the free meet-semilattice on a set $S$ is the meet-semilattice of finite subsets of $S$ ordered by reverse inclusion.
The free frame on a set is then just the frame of downsets on this free meet-semilattice.

\subsection{Compact frames}

An element $a$ of a frame $L$ is said to be \emph{$\kappa$-Lindelöf} if whenever $a \le \bigvee_{\alpha \in I} c_\alpha$,
then $a \le \bigvee_{\alpha \in J} c_\alpha$ for some $\kappa$-set $J \subseteq I$. If $\kappa = \aleph_0$, we say $a$ is \emph{compact}
and if $\kappa = \aleph_1$ we simply say $a$ is Lindelöf.

A frame $L$ is said to be $\kappa$-Lindelöf (resp.\ compact/Lindelöf) if its top element is.
A frame is called \emph{Noetherian} if all of its elements are compact.
Complemented elements of $\kappa$-Lindelöf frames are $\kappa$-Lindelöf elements.
If an element of a subframe is $\kappa$-Lindelöf in the parent frame, then it is $\kappa$-Lindelöf in the subframe.
In particular subframes of $\kappa$-Lindelöf frames are $\kappa$-Lindelöf.

A dense frame homomorphism from a regular frame to a compact frame is injective.
In a regular Lindelöf frame $\prec$ interpolates and so, assuming dependent choice, regular Lindelöf frames are completely regular.

The forgetful functor from the category of compact completely regular frames to the category of completely regular frames has a right adjoint,
the compact completely regular coreflection $\StoneCech$, which the Stone-Čech compactification functor.
A completely regular ideal $I$ on $L$ is an ideal for which whenever $a \in I$ there is a $b \in I$ such that $a \prec\prec b$.
The completely regular ideals on $L$ form a subframe of $\ideal L$ isomorphic to $\StoneCech L$. The action of $\StoneCech$ on
morphisms is inherited from $\ideal$. The counit of the adjunction is given by taking joins and is a dense surjection.

In general, any dense surjection from a compact regular frame to a frame $L$ is known as a \emph{compactification} of $L$.
If $k: M \to L$ is a compactification, the well below relation on $M$ induces a relation $\lhd$ on $L$ (the \emph{strongly below} relation)
which satisfies the following axioms \cite{StrongInclusions}: %
\begin{enumerate}
 \item $x \le y \lhd z \le w \implies x \lhd w$
 \item $\lhd$ is a sublattice of $L \times L$
 \item $x \lhd y \implies x \prec y$
 \item $x \lhd y$ implies there is a $z \in L$ such that $x \lhd z \lhd y$ (i.e.\ $\lhd$ interpolates)
 \item $x \lhd y \implies y^* \lhd x^*$
 \item Every element in $L$ is the a join of elements strongly below it.
\end{enumerate}
Any relation which satisfies these axioms is called a \emph{strong inclusion}. The strong inclusions on $L$ and the compactifications of $L$
are in one-to-one correspondence and the associated compactification can be recovered in the exact same way as the above construction
of the Stone-Čech compactification, but using $\lhd$ instead of $\prec\prec$ (which is a strong inclusion on completely regular frames).

A zero-dimensional frame $L$ admits a universal zero-dimensional compactification which is given by
$\ideal B L$ where $B L$ is the Boolean algebra of complemented elements of $L$ (\cite{ZeroDimensionalCompactifications}).

\subsection{Cozero elements}

An element of a frame is said to be \emph{cozero} if it is a countable join of elements completely below it (see \cite{PseudoCozeroPartFrame}).
Assuming countable choice, the cozero elements of a frame $L$ form a completely regular sub-$\sigma$-frame of $L$, which we denote by $\Coz L$.
Frame homomorphisms preserve cozero elements and so $\Coz\!$ becomes a functor.

The functor $\Coz\!$ from completely regular frames to completely regular $\sigma$-frames has the left adjoint $\h$, the $\sigma$-ideal functor.
This adjunction induces an equivalence between the categories of completely regular Lindelöf frames and completely regular $\sigma$-frames
and the map $\lambda: \h\Coz L \to L$ sending $\sigma$-ideals to their joins gives the completely regular Lindelöf coreflection of a completely
regular frame. The map $\lambda$ is a dense surjection from a regular Lindelöf frame and is called a \emph{Lindelöfication} by analogy to
compactifications.

\subsection{Coherent and \texorpdfstring{$\kappa$}{κ}-coherent frames}

The join-semilattice of $\kappa$-Lindelöf elements of a frame $L$ is written $K_\kappa(L)$.
A frame is said to be \emph{$\kappa$-coherent} if $K_\kappa(L)$ is a sub-$\kappa$-frame of $L$ which
generates $L$ under arbitrary joins. A frame homomorphism between $\kappa$-coherent frames is
called \emph{$\kappa$-proper} if it maps $\kappa$-Lindelöf elements to $\kappa$-Lindelöf elements.
When $\kappa = \aleph_0$ we speak of coherent frames and proper maps and write $K_{\aleph_0}(L)$ as $K(L)$.

The $\kappa$-coherent frames are precisely the frames of the form $\h_\kappa L$ for some $\kappa$-frame $L$. %
Compact zero-dimensional frames are coherent and completely regular $\kappa$-Lindelöf frames are $\kappa$-coherent for $\kappa \ge \aleph_1$.

\subsection{Biframes}

A \emph{biframe} $L$ is a triple $(L_0, L_1, L_2)$ where $L_0$ is a frame and $L_1$ and $L_2$ are subframes of $L_0$
such that $L_1 \cup L_2$ generates $L_0$. The frame $L_0$ is called the \emph{total part} of $L$, while $L_1$ and
$L_2$ are known as the first and second parts respectively. We often use the indices $i$ and $k$ to refer to
elements of $\{1,2\}$ so we can treat the first and second parts uniformly.

A biframe homomorphism $f: L \to M$ is a frame homomorphism $f_0: L_0 \to M_0$ such that elements in the $i^\text{th}$ part of $L$ map to elements
in the $i^\text{th}$ part of $M$. We use $f_i: L_i \to M_i$ to denote the restriction of $f_0$ to the $i^\text{th}$ part.

A bitopological space (or \emph{bispace}) is a set equipped with two topologies. A \emph{bicontinuous map} between bispaces is a function
that is pairwise continuous with respect to first topologies and the second topologies. There is a dual adjunction between the categories of
biframes and bitopological spaces (see \cite{Biframes}).

A biframe is said to be compact/$\kappa$-Lindelöf if its total part is and a biframe homomorphism is dense/injective if its total part is dense/injective.
A biframe map is surjective if its restrictions to each part are surjective.

The biframe of ideals on a biframe $L$ is given by $\ideal L = (M_0, M_1, M_2)$ where $M_0 = \ideal L_0$ and $M_i$
consists of the ideals in $M_0$ generated by elements of $L_i$. Biframes of $\kappa$-ideals can be defined similarly.

The rather below relation is replaced by two relations, $\prec_1$ and $\prec_2$ on $L_1$ and $L_2$ respectively.
If $x, y \in L_i$ then $x \prec_i y$ if and only if there is a $z \in L_k$ ($k \ne i$) such that $x \wedge z = 0$, but $y \vee z = 1$.
Let $x^\bullet = \bigvee\{y \in L_k \mid y \wedge x = 0\}$ for $x \in L_i,\, k \ne i$. Then $x \prec_i y$ if and only if $x^\bullet \vee y = 1$.
A biframe is regular if every element in each part $L_1, L_2$ is a join of elements rather below it and complete regularity is defined similarly.
A biframe is zero-dimensional if each part $L_i$ is generated by elements with complements in $L_k$ ($k \ne i$) and is Boolean if every
element in each part has a complement in the other part.

A strong inclusion of biframes is defined completely analogously to strong inclusions of frames (see \cite{BiframeCompactifications} for details).
As before, the biframe compactifications of a biframe $L$ correspond bijectively with the strong inclusions on $L$.

A $\sigma$-biframe $A$ is a triple of $\sigma$-frames $(A_0, A_1, A_2)$ such that $A_1, A_2$ are sub-$\sigma$-frames of $A_0$ which
together generate $A_0$ (\cite{BisigmaFrames}). The cozero part $\Coz L$ of a biframe $L$ is defined as $(A_0, A_1, A_2)$ where
$A_i = (\Coz L_0) \cap L_i$ and $A_0 = \langle A_1 \cup A_2\rangle$, the sub-$\sigma$-frame of $L_0$ generated by $A_1$ and $A_2$.
The biframe $\h \Coz L$ with the join map $\lambda: \h\Coz L \to L$ is the universal Lindelöfication of a completely regular biframe $L$,
as in the frame case.

\section{Frames of congruences}\label{section:congruences}

\subsection{Congruences}

For any class of (potentially infinitary) algebraic structures, one might examine the nature of the congruences on those structures.

\begin{definition}
 A \emph{congruence} on an algebraic structure $A$ is an equivalence relation on $A$ that is also a subalgebra of $A \times A$.
\end{definition}

\begin{example}
 If $A$ is a meet-semilattice then a congruence on $A$ is an equivalence relation on $A$ for which whenever $x \sim y$ and $x' \sim y'$,
 we have $x \wedge x' \sim y \wedge y'$.
\end{example}

\begin{example}
 If $A$ and $B$ are algebraic structures and $f: A \to B$ is a homomorphism of these structures then $\ker f = \{(x,y) \in A\times A \mid f(x) = f(y)\}$
 is a congruence on $A$. This congruence is known as the \emph{kernel} of $f$.
\end{example}

Congruences are important since we may endow the quotient set of $A$ by a congruence $C$ with the appropriate structure to
form the \emph{quotient object} $A/C$ so that the canonical map sending $x \in A$ to $[x] \in A/C$ is a surjective homomorphism. In this
way we may move freely between congruences on $A$, quotient algebras of $A$ and (isomorphism classes of) surjective homomorphisms from $A$.

\begin{remark}\label{rem:factor_through_quotients}
 If $f: A \to B$ is a homomorphism, $C$ is a congruence on $A$ and $f(x) = f(y)$ for all $(x,y) \in C$, then $f$ factors through
 the quotient map $q: A \to A/C$.
\end{remark}

Since quotient maps are surjections, they are immediately seen to be epimorphisms, but we can say more.
\begin{definition}
 An \emph{extremal epimorphism} is a morphism $e: A \to B$ such that for any factorisation $e = mf$, where $m$ is a monomorphism,
 $m$ is in fact an isomorphism.
\end{definition}
\begin{lemma}\label{lem:extremal_epi_surjective}
 Let $\A$ denote the category of some variety of algebras. The extremal epimorphisms in $\A$ are precisely the surjective homomorphisms.
\end{lemma}
\begin{proof}
 If $e$ is a surjection and $e = mf$ then $m$ is also a surjection. So if $m$ is additionally a monomorphism, it is both a surjection and
 an injection and thus a bijection. Bijective homomorphisms are isomorphisms and thus surjective homomorphisms are extremal epimorphisms.
 Any homomorphism $e: A \to B$ can be factored as $ie'$ where $e'$ is a surjection onto the image of $e$ and $i$ is the inclusion of the image into $B$.
 If $e$ is an extremal epimorphism, then the inclusion $i$ is a bijection and so $e$ is surjective.
\end{proof}

The set of all congruences on $A$, denoted here by $\C A$, may be ordered by set inclusion. Since $\C A$ is closed under arbitrary intersections,
it is a complete sub-meet-semilattice of $\mathcal{P}(A\times A)$ and thus forms a complete lattice. The top of this lattice is $1 = A\times A$
and the bottom is $0 = \{(x,x) \mid x \in A\}$. Meets in $\C A$ are simply intersections, but joins arise by
\emph{generating a congruence} from the union.

\begin{definition}
 If $S \subseteq A\times A$ then the smallest congruence in $\C A$ containing $S$ is called the \emph{congruence generated by $S$} and
 is denoted $\langle S\rangle$.
 If $S$ is a singleton $\{(x,y)\}$, we say $\langle S \rangle = \langle(x,y)\rangle$ is the \emph{principal}
 congruence generated by $(x,y)$.
\end{definition}

\begin{remark}\label{rem:lattice_of_quotients}
 The quotients of $A$ may also be naturally ordered such that $A/C \le A/D$ if the quotient map $q: A \to A/C$ factors through the quotient map $r: A \to A/D$.
 By \cref{rem:factor_through_quotients} above, this is just the inverse order to that induced by the ordering on $\C A$.
 
 Furthermore, the induced map $h: A/D \to A/C$ is a surjection and every surjective homomorphism from $A/D$ is of this form.
 Thus, the poset of quotients of $A/D$ is isomorphic to the poset of the quotients of $A$ that are less than or equal to $A/D$.
\end{remark}

It is often useful to define algebraic structures as quotients of free structures. This is known as defining the structure by
\emph{generators and relations}.
\begin{example}
 The group $D_6$ of symmetries of a regular hexagon could be described as follows. Let $F_2$ denote the free group on the set $\{r,t\}$.
 Let $C$ be the congruence on $F_2$ generated by $(r^6, 1)$, $(t^2, 1)$ and $(rt, tr^{-1})$. Then $D_6 \cong F_2/C$ and we may write
 this as $D_6 = \langle r,t \mid r^6 = 1, t^2 =1, rt = tr^{-1}\rangle$.
\end{example}
Any algebraic structure $A$ can be defined in this way by generating a free structure $F$ from all its elements and quotienting out by the
kernel of the obvious morphism from $F$ to $A$.

\subsection{Lattice congruences}

We now restrict ourselves to varieties of algebra that have binary operations, $\wedge$ and $\vee$, which satisfy the lattice axioms.
In this case, congruences can be thought of as specifying a collection of intervals such that all the elements in each interval collapse
to a single element. This is made precise in the following lemma.

\begin{lemma}\label{lem:congruence_intervals} %
 Let $L$ be a lattice, let $C$ be a congruence on $L$ and suppose that $(a,b) \in C$. Then $(a\wedge b, a\vee b) \in C$ and
 furthermore, if $a \le c \le b$, then $(a, c) \in C$ and $(c, b) \in C$.
\end{lemma}

\begin{proof}
 Suppose $(a,b) \in C$. Then $(a,b) \wedge (b,b) = (a \wedge b, b) \in C$ and  $(a,b) \vee (a,a) = (a, a \vee b) \in C$.
 So $(a \wedge b, b), (b, a), (a, a \vee b) \in C$. Thus, by transitivity, $(a \wedge b, a \vee b) \in C$.
 
 If $a \le c \le b$, then $(a, b) \vee (c, c) = (c, b) \in C$ and $(a,b) \wedge (c, c) = (a, c) \in C$.
\end{proof}

In 1942, Funayama and Nakayama showed that the congruence lattices of lattices turn out to be (finitely) distributive \cite{Distributivity}.
In contrast, distributivity of congruence lattices of complete lattices fails spectacularly: in 1988, Grätzer showed that
\emph{every} complete lattice is isomorphic to the congruence lattice of a complete lattice \cite{CompleteCongruence}.

The distributivity result of congruence lattices \emph{does} generalise in the case of distributive lattices and this is the main
theorem presented in this section. First observe that since directed joins of congruences of finitary algebraic structures are simply unions,
finite meets distribute over directed joins of lattice congruences. This, combined with the finite distributivity of \cite{Distributivity}, implies
the frame distributivity law. So the congruence lattice of a lattice, and in particular of a distributive lattice, is a \emph{frame}.
We will see that this result extends from distributive lattices to arbitrary $\kappa$-frames and to frames themselves.

\subsection{Frame congruences}\label{subsec:frame_congruences}

In this subsection we take a closer look at congruences of frames.
If $C$ is a frame congruence on $L$ and $(a_\alpha, b_\alpha) \in C$ for all $\alpha \in I$ then
$(\bigvee_{\alpha \in I} a_\alpha, \bigvee_{\alpha \in I} b_\alpha) \in C$. In particular, every equivalence class $[a]$ in $L/C$
has a largest element: $\bigvee [a]$.

The map $\nu_C(a) = \max\ [a]$ is called a \emph{nucleus}. Given $\nu_C$ we can recover the congruence $C$ as
$\{(a,b) \in L\times L \mid \nu_C(a) = \nu_C(b)\}$ and so nuclei and congruences are in one-to-one correspondence.
Nuclei are characterised in the following lemma.
\begin{lemma}\label{lem:nuclei} %
 A function $\nu: L \to L$ on a frame $L$ is a \emph{nucleus} if and only if it satisfies the following.
 \begin{enumerate}
  \item {\makebox[5cm][l]{$\nu(x) \ge x$} ($\nu$ is inflationary)}
  \item {\makebox[5cm][l]{$\nu \circ \nu = \nu$} ($\nu$ is idempotent)}
  \item {\makebox[5cm][l]{$\nu(x \wedge y) = \nu(x) \wedge \nu(y)$} ($\nu$ preserves binary meets)} \label{property:meet_preserving}
 \end{enumerate}
\end{lemma}
\begin{proof}
 ($\Rightarrow$) It is straightforward to check that a nucleus $\nu$ is monotone, inflationary and idempotent.
 It remains to show that $\nu(x) \wedge \nu(y) \le \nu(x \wedge y)$. In the associated congruence, $x \sim \nu(x)$ and $x \sim \nu(y)$.
 Thus, $x \wedge y \sim \nu(x) \wedge \nu(y)$ and so $\nu(x \wedge y) = \nu(\nu(x) \wedge \nu(y))$ from which the desired inequality follows.
 
 ($\Leftarrow$) Let $C = \{(a,b) \in L\times L \mid \nu_C(a) = \nu_C(b)\}$. This is obviously an equivalence relation. We show it is a congruence.
 Closure under binary meets follows from property \ref{property:meet_preserving}. Now take $(a_\alpha)_{\alpha \in I}$ in $L$.
 Since $\nu$ is inflationary and monotone, $\nu\left(\bigvee_{\alpha \in I} a_\alpha\right) \le \nu\left(\bigvee_{\alpha \in I} \nu(a_\alpha)\right)$.
 But again by monotonicity, $\bigvee_{\alpha \in I} \nu(a_\beta) \le \nu\left(\bigvee_{\alpha \in I} a_\alpha\right)$.
 Applying $\nu$ to both sides and using idempotence, we obtain
 $\nu\left(\bigvee_{\alpha \in I} \nu(a_\beta)\right) \le \nu\left(\bigvee_{\alpha \in I} a_\alpha\right)$.
 Thus, $\bigvee_{\alpha \in I}a_\alpha \sim \bigvee_{\alpha \in I} \nu(a_\alpha)$ in $C$ and so, by transitivity, $C$ is closed under joins.
\end{proof}

The usual order on congruences corresponds to the pointwise order on nuclei. The pointwise meet of a family of nuclei can be shown to
again be a nucleus and so meets can always be calculated pointwise, but joins are more complicated.

Since every equivalence class in a quotient frame has its maximum element as a canonical representative, quotient frames can be
viewed as a subposet of the original frame and one can show that this poset is closed under meets in the parent frame.

Restricting the codomain of the nucleus to this subset turns the nucleus into a surjective frame map --- the associated quotient map.
Also notice that we can recover a nucleus from a surjective frame map $h$ as $h_*h$ where $h_*$ is the right adjoint of $h$.

\subsection{Frame congruences and subspaces}\label{subsec:congruences_and_subspaces}

Since frames were originally motivated as `generalised' topological spaces, it is interesting to consider the relationship between
frame congruences and their spatial analogue.

By \cref{lem:extremal_epi_surjective}, quotient maps are the extremal epimorphisms in $\Frm$.
These are analogous to \emph{extremal monomorphisms} in $\Top$.
\begin{lemma}
 The extremal monomorphisms in $\Top$ are the subspace embeddings.
\end{lemma}
\begin{proof}
 Recall that the monomorphisms and epimorphisms in $\Top$ are the injections and surjections respectively.
 Suppose $m: X \to Y$ is an embedding and $m = fe$ where $e$ is an epimorphism. Since $m$ is an injection, $e$ is a bijection.
 But now recall that an embedding is \emph{initial} --- the topology on its domain is induced by the topology on its codomain.
 Since $m$ is initial, so is $e$. Thus $e$ is a homeomorphism and $m$ is an extremal monomorphism.
 
 Suppose $m: X \to Y$ is an extremal monomorphism. We can factorise $m$ as $im'$ where $m'$ is the restriction of $m$ to its image
 and $i$ is an embedding. Since $m'$ is an epimorphism and $m$ is extremal, $m'$ is a homeomorphism. Thus, $m$ is an embedding.
\end{proof}

We often think of quotient frames as if they were subspaces and call them \emph{sublocales}.
If $L$ is a spatial frame, its sublocales and subspaces may be compared directly.
The remaining results in this section are taken from \cite[pp.~7--8]{PicadoPultr}.

Let $X$ be a topological space and $A \subseteq X$. Let $i_A: A \hookrightarrow X$ be the inclusion map.
The corresponding frame map $\Omega(i_A): \Omega(X) \to \Omega(A)$ is given by $\Omega(i_A)(U) = U \cap A$.
Since this map is surjective, we can indeed view $\Omega(A)$ as a quotient of $\Omega(X)$. On the other hand,
in \cref{subsec:dense_congruences} %
we will see that there are often additional non-spatial sublocales that do not correspond to any subspace.

The congruence associated with $A$ is $E_A = \{(U,V) \in \Omega(X)\times\Omega(X) \mid U\cap A = V\cap A\}$.
It is easy to see that if $A \supseteq B$ then $E_A \le E_B$ and more generally that $E_{A \cup B} = E_A \wedge E_B$.
(Recall that congruences have the reverse order to sublocales.)
However, it is generally not true that $E_{A\,\cap B} = E_A \vee E_B$ (see \cref{ex:smallest_dense_sublocale_of_reals} in \cref{subsec:dense_congruences}).

The following example shows that it is possible that $E_A \le E_B$, but $B \nsubseteq A$.
\begin{example}
 Let $X$ be the sobrification of $\N$ with the cofinite topology.
 That is, $X = \N \cup \{\omega\}$ and the non-empty open sets are the cofinite sets which contain $\omega$.
 The subsets $\N$ and $X$ induce the same congruence since no two open sets differ merely in whether they contain $\omega$.
\end{example}

Fortunately, there is a large class of spaces, where the above pathology is ruled out.
In particular, this class includes every $T_1$ space.
\begin{definition}\label{def:T_D}
 A topological space $X$ is $T_D$ if for every $x \in X$ there is an open $U \ni x$ such that $U \setminus \{x\}$ is still open.
\end{definition}
\begin{lemma}
 A space $X$ is $T_D$ if and only if, for all $A, B \subseteq X$, $E_A \le E_B \implies B \subseteq A$.
\end{lemma}
\begin{proof}
 ($\Rightarrow$) Suppose $B \nsubseteq A$. Then there is an $x \in A \setminus B$. Let $U$ be an open set containing $x$ with
 $V = U \setminus \{x\}$ also open. Then $(U, V) \in E_B \setminus E_A$, so $E_A \not\le E_B$.
 
 ($\Leftarrow$) Take $x \in X$. Since $X \nsubseteq X\setminus\{x\}$, we know $E_{X\setminus\{x\}} \not\le E_X$.
 So there is some pair of distinct open sets $(U,V) \in E_{X\setminus\{x\}}$. That is, $U \cap (X\setminus\{x\}) = V \cap (X\setminus\{x\})$,
 but $U \ne V$. Thus, one of $U$ and $V$, say $U$, contains $x$ and then $V = U \setminus \{x\}$.
\end{proof}

\subsection{Closed and open congruences}\label{subsec:closed_open_congs}

From now on we will restrict our consideration to congruences of frames and $\kappa$-frames.
Motivated by our spatial analogy, we look for congruences which we could call open or closed.

When considering a subspace $S$ of a topological space $(X,\tau)$, we usually think of $S$ as now being `everything that exists' ---
in the subspace, any open set that contains $S$ may be identified with $S$. If $S$ is open in $X$, it is an element of $\tau$ and we can
define a congruence $\Delta_S$ on $\tau$ which identifies $S$ with the top element of the frame, $\Delta_S = \langle(S,1)\rangle$.

If $S$ is a closed subspace, we cannot refer to it directly as an element of $\tau$, but its complement $T = X \setminus S$
is open and so $T \in \tau$. Any open subset of $T$ completely misses $S$ and may be identified with
the empty set. We obtain the congruence $\nabla_T = \langle(0,T)\rangle$.

Both of the above congruences are principal congruences of particularly simple forms.
We are thus lead to define open and closed congruences more generally.

\begin{definition}
 If $L$ is a $\kappa$-frame and $a \in L$, an \emph{open congruence} on $L$ is a congruence of the form $\Delta_a = \langle(a,1)\rangle$
 and a \emph{principal closed congruence} is one of the form $\nabla_a = \langle(0,a)\rangle$.
\end{definition}

It is sometimes useful to have a more explicit description of open and principal closed congruences.
\begin{lemma}\label{lem:open_and_closed_congruences} %
 $\nabla_a = \{(x,y) \,\mid\, x \vee a = y \vee a\}$ and $\Delta_a = \{(x,y) \,\mid\, x \wedge a = y \wedge a\}$.
\end{lemma}
\begin{proof}
 Let $A = \{(x,y) \,\mid\, x \vee a = y \vee a\}$ and $B = \{(x,y) \,\mid\, x \wedge a = y \wedge a\}$.
 Clearly these are both equivalence relations. Furthermore, $A$ and $B$ are closed under binary meets by
 the distributivity of $L$ and commutativity/idempotence respectively. Similarly, $A$ and $B$ are closed
 under $\kappa$-joins by commutativity/idempotence and distributivity respectively.
 
 Since $(0,a) \in A$, we have $\nabla_a \subseteq A$.
 
 On the other hand, suppose $x,y \in L$ are such that $x \vee a = y \vee a$.
 Now $(x,x) \in \nabla_a$ by reflexivity and so $(x \vee 0,x \vee a) = (x, x \vee a) \in \nabla_a$.
 Similarly, $(y \vee a, y) \in \nabla_a$. But $x \vee a = y \vee a$, so by transitivity $(x,y) \in \nabla_a$.
 Thus, $A \subseteq \nabla_a$ and so $\nabla_a = A$.
 
 The proof that $\Delta_a = B$ is similar.
\end{proof}
\begin{corollary} \label{cor:nabla_mono}
 The functions $\nabla_\bullet: a \mapsto \nabla_a$ and $\Delta_\bullet: a \mapsto \Delta_a$ are injective.
\end{corollary}

The following lemma confirms suspicions about what `identifying an element $a$ with the top or bottom element' might mean.
\begin{lemma}\label{lem:open_closed_lowerset_and_upperset}
 If $a \in L$ then $L/\Delta_a \cong \;\downarrow\! a$ and $L/\nabla_a \cong \;\uparrow\! a$.
\end{lemma}
\begin{proof}
 First notice that $\downarrow\!a$ and $\uparrow\!a$ are indeed $\kappa$-frames with the order which they inherit from $L$.
 (They are \emph{not} sub-$\kappa$-frames of $L$, however, since they may fail to contain $1$ and $0$ respectively.)
 
 The maps $(\bullet \wedge a): L \to  \;\downarrow\! a$ and $(\bullet \vee a): L \to  \;\uparrow\! a$
 are surjective $\kappa$-frame homomorphisms with kernels $\Delta_a$ and $\nabla_a$ respectively.
\end{proof}

We can generalise \cref{lem:open_and_closed_congruences} to give an expression for the binary join of one of these special congruences
and an arbitrary congruence \cite{FrithCong}.
\begin{lemma}\label{lem:joins_with_principal}
 If $a \in L$ and $C \in \C L$ then $\nabla_a \vee C = \{(x,y) \,\mid\, (x \vee a, y \vee a) \in C\}$ and
 $\Delta_a \vee C = \{(x,y) \,\mid\, (x \wedge a, y \wedge a) \in C\}$.
\end{lemma}
\begin{proof}
 As for \cref{lem:open_and_closed_congruences}.
\end{proof}

Every principal congruence can be expressed in terms of open and principal closed congruences. Consider $\langle(a,b)\rangle$.
By \cref{lem:congruence_intervals} we may assume $a \le b$ without loss of generality and then apply the following lemma.
\begin{lemma} \label{lem:principal_congruence}
 Suppose $a,b \in L$ and $a \le b$. Then $\langle(a,b)\rangle = \nabla_b \wedge \Delta_a$.
\end{lemma}
\begin{proof}
 Since $0 \le a \le b$, we get $(a,b) \in \nabla_b$ by \cref{lem:congruence_intervals}. Similarly,
 $(a,b) \in \Delta_a$. Thus, $(a,b) \in \nabla_b \wedge \Delta_a$ and so $\langle(a,b)\rangle \subseteq \nabla_b \wedge \Delta_a$.
 
 Now suppose $(x,y) \in \nabla_b \wedge \Delta_a$. Then $x \vee b = y \vee b$ and $x \wedge a = y \wedge a$.
 Notice that $(x,x) \wedge (a,b) = (x \wedge a, x \wedge b) \in \langle(a,b)\rangle$ and similarly $(y \wedge a, y \wedge b) \in \langle(a,b)\rangle$.
 But $x \wedge a = y \wedge a$ then gives $(x \wedge b, y \wedge b) \in \langle(a,b)\rangle$ by transitivity.
 
 We then have $(x\wedge y, x\wedge y) \vee (x \wedge b, y \wedge b) = ((x\wedge y)\vee (x \wedge b), (x\wedge y) \vee (y \wedge b)) \in \langle(a,b)\rangle$.
 But by distributivity of $L$,
 $((x\wedge y)\vee (x \wedge b), (x\wedge y) \vee (y \wedge b)) = (x\wedge (y\vee b), y\wedge(x\vee b))$.
 And using $y\vee b = x \vee b$ and absorption we see this is in turn equal to $(x, y)$ and so $(x,y) \in \langle(a,b)\rangle$ and
 $\nabla_b \wedge \Delta_a \subseteq \langle(a,b)\rangle$ as required.
\end{proof}

\begin{corollary} \label{cor:congruence_lattice_generators}
 Any $A \in \C L$ may be written as $A = \bigvee \{\nabla_b \wedge \Delta_a \mid (a,b) \in A,\, a \le b\}$.
\end{corollary}

\begin{corollary} \label{cor:nabla_delta_complements}
 $\nabla_a$ and $\Delta_a$ are complements in $\C L$.
\end{corollary}
\begin{proof}
 It is clear that $\nabla_a \vee \Delta_a = \langle(0,a)\rangle \vee \langle(a,1)\rangle = 1$. Then the above lemma gives us
 $\nabla_a \wedge \Delta_a = \langle(a,a)\rangle = 0$.
\end{proof}

\begin{remark}
Once we show $\C L$ is a frame, we will be able to interpret the above corollaries to mean that $\C L$ is zero-dimensional.
\end{remark}

The next lemma is an analogue of the familiar spatial results involving the behaviour of open and closed subspaces under unions and intersections.
\emph{Remember that the lattice of congruences is ordered in the reverse order of the lattice of sublocales.}
Joins in the congruence lattice correspond to meets of sublocales and vice versa.

\begin{lemma} \label{lem:joins_and_meets_nabla} %
 Suppose $a,b \in L$ and $(c_\alpha)_{\alpha \in I}$ is a family of cardinality less than $\kappa$.
 Then the following identities hold:
 \begin{enumerate}[i)]
  \item $\nabla_a \wedge \nabla_b = \nabla_{a\wedge b}$
  \item $\Delta_a \vee \Delta_b = \Delta_{a\wedge b}$
  \item $\bigvee_{\alpha \in I} \nabla_{c_\alpha} = \nabla_{\bigvee_{\alpha \in I} c_\alpha}$
  \item $\bigwedge_{\alpha \in I} \Delta_{c_\alpha} = \Delta_{\bigvee_{\alpha \in I} c_\alpha}$
 \end{enumerate}
\end{lemma}
\begin{proof}
 By \cref{lem:congruence_intervals}, it is clear that $\nabla_\bullet$ is order-preserving. Thus $\nabla_a \wedge \nabla_b \ge \nabla_{a\wedge b}$.
 Now let $(x,y) \in \nabla_a \wedge \nabla_b$. Then $x \vee a = y \vee a$ and $x \vee b = y \vee b$ and so
 $x \vee (a\wedge b) = (x \vee a) \wedge (x \vee b) = (y \vee a) \wedge (y \vee b) = y \vee (a\wedge b)$. Thus $(x,y) \in \nabla_{a\wedge b}$ and
 $\nabla_a \wedge \nabla_b \le \nabla_{a\wedge b}$.
 
 Again by \cref{lem:congruence_intervals}, $\Delta_\bullet$ is order-reversing and so $\Delta_a \vee \Delta_b \le \Delta_{a\wedge b}$.
 To show the other inequality, we need only show that $(a\wedge b,1) \in \Delta_a \vee \Delta_b$.
 But this is clear since $(a,1) \in \Delta_a \le \Delta_a \vee \Delta_b$ and $(b,1) \in \Delta_b \le \Delta_a \vee \Delta_b$.
 
 The non-trivial direction of (iii) follows from the observation that $(0, \bigvee_{\alpha \in I} c_\alpha) \in \bigvee_{\alpha \in I} \nabla_{c_\alpha}$
 since $(0, c_\alpha) \in \bigvee_{\alpha \in I} \nabla_{c_\alpha}$ for all $\alpha \in I$.
 
 For the non-trivial direction of (iv), suppose $(x,y) \in \Delta_{c_\alpha}$ for all $\alpha \in I$. Then $x \wedge c_\alpha = y \wedge c_\alpha$ and so
 $x \wedge \bigvee_{\alpha \in I} c_\alpha = \bigvee_{\alpha \in I} x \wedge c_\alpha = \bigvee_{\alpha \in I} y \wedge c_\alpha =
  y \wedge \bigvee_{\alpha \in I} c_\alpha$.
 Thus, $(x,y) \in \Delta_{\bigvee_{\alpha \in I} c_\alpha}$ as required.
\end{proof}

\begin{remark}
 Once we show that $\C L$ is a frame, it will follow from the above lemma that $\nabla_\bullet : L \to \C L$ is a $\kappa$-frame homomorphism.
\end{remark}

In the case of $\kappa$-frames, principal closed congruences are only closed under joins of cardinality less than $\kappa$.
However, it is often useful to consider arbitrary joins of principal closed congruences.

\begin{definition}
 A congruence on a $\kappa$-frame is called \emph{generalised closed} if it is a join of principal closed congruences.
 If $I$ is a $\kappa$-ideal on a $\kappa$-frame, we define the generalised closed congruence $\widetilde{\nabla}_I = \bigvee_{a \in I} \nabla_a$.
\end{definition}

We can use \cref{lem:joins_and_meets_nabla} to reduce any join of principal closed congruences to a $\kappa$-directed join. Also note that if a
principal closed congruence $\nabla_a$ is included in such a join, we might as well include any smaller principal closed congruence.
But a $\kappa$-directed lower set is nothing but a $\kappa$-ideal. So every generalised closed congruence is of the form
$\widetilde{\nabla}_I$ for some $\kappa$-ideal $I$ on $L$.

There are analogues of \cref{lem:open_and_closed_congruences,lem:joins_with_principal} for generalised closed congruences.

\begin{lemma}\label{lem:joins_with_generalised_closed}
 If $L$ is a $\kappa$-frame, $C \in \C L$ and $I$ is a $\kappa$-ideal on $L$, then
 $\widetilde{\nabla}_I \vee C = \{(x,y) \mid (x \vee i, y \vee i) \in C \text{ for some $i \in I$}\}$.
 In particular, $\widetilde{\nabla}_I = \{(x,y) \mid x \vee i = y \vee i \text{ for some $i \in I$}\}$.
\end{lemma}
\begin{proof}
 As for \cref{lem:open_and_closed_congruences}.
\end{proof}

\begin{lemma}\label{lem:nabla_tilde_injective}
 The map $\widetilde{\nabla}_\bullet: \h_\kappa L \to \C L$ is injective and preserves finite meets and arbitrary joins.
\end{lemma}
\begin{proof}
 Preservation of joins comes from the definition of $\widetilde{\nabla}_\bullet$, while injectivity follows from \cref{lem:joins_with_generalised_closed}.
 Preservation of finite meets is shown in the same way as for $\nabla_\bullet$ in \cref{lem:joins_and_meets_nabla}.
\end{proof}

The composition of $\widetilde{\nabla}_\bullet$ with its right adjoint gives a map $\cl = \widetilde{\nabla}_\bullet \circ (\widetilde{\nabla}_\bullet)_*$,
which is of some interest. The map $\cl: \C L \to \C L$ assigns a congruence to the largest generalised closed congruence which it contains.
Thus, again recalling the inverted order of $\C L$, it is an analogue of the topological closure operator and
it is easy to see that it is monotone, deflationary and idempotent and that it preserves finite meets, as expected.
\begin{definition}
 If $C$ is a congruence, $\cl(C)$ is called the \emph{(generalised) closure} of $C$.
\end{definition}

We take a look at how closed congruences behave in quotient $\kappa$-frames.
\begin{lemma}\label{lem:closed_quotient_of_quotient} %
 Let $L$ be a $\kappa$-frame and $L/C$ a quotient of $L$. If $[a] \in L/C$, then $(L/C)/\nabla_{[a]} \cong L / (C \vee \nabla_a)$
 in a natural way.
\end{lemma}
\begin{proof}
 Consider the congruence $A$ on $L$ induced by $(L/C)/\nabla_{[a]}$.
 By \cref{lem:open_and_closed_congruences}, we have that $(x,y) \in A$ if and only if $[x] \vee [a] = [y] \vee [a]$.
 That is, if $(x\vee a, y \vee a) \in C$. But this just means that $(x,y) \in C \vee \nabla_a$ by \cref{lem:joins_with_principal}.
\end{proof}
\begin{corollary}\label{cor:closed_congruences_in_quotient} %
 If $a,b \in L$, then $\nabla_a \vee C = \nabla_b \vee C$ if and only if $(a,b) \in C$.
\end{corollary}
\begin{remark}\label{rem:closed_congruences_preserved_by_quotient}
 Recall from \cref{rem:lattice_of_quotients} that there is a correspondence between quotients of $L/C$ and quotients
 of $L$ by congruences lying above $C$. The above shows that this correspondence preserves closed congruences.
\end{remark}

\begin{corollary}\label{cor:closure_in_quotient} %
 Suppose $A \ge C \in \C L$. The closure of $A$ in $\C(L/C)$ corresponds to $C \vee \cl(A)$ in $\C L$.
\end{corollary}

\subsection{Dense congruences}\label{subsec:dense_congruences}

We are now in a position to discuss what it means for a congruence to be dense.

\begin{definition}
 Suppose $C,D \in \C L$ and $D \le C$. We say $C$ is \emph{dense in $D$} if $\cl(C) \le D$.
 We say $C$ is \emph{dense} if it is dense in $0$.
\end{definition}

This definition is quickly seen to be compatible with our notion of dense maps.
\begin{lemma}\label{lem:dense}
 A congruence $C \in \C L$ is dense if and only if the corresponding quotient map $h: L \twoheadrightarrow L/C$ is a dense $\kappa$-frame homomorphism.
\end{lemma}
\begin{proof}
 Simply observe that $h(a) = 0 \iff (0,a) \in C \iff \nabla_a \le C$.
\end{proof}
\begin{remark}
 An element $a$ of a frame $L$ is called dense if $a^* = 0$. An element $a \in L$ is dense if and only if $\Delta_a$
 is a dense congruence.
\end{remark}

The following generalisation then follows from \cref{cor:closure_in_quotient}.
\begin{corollary}\label{cor:dense_quotients_congruence}
 A congruence $C \in \C L$ is dense in $D$ (where $D \le C$) if and only if the canonical map $h: L/D \to L/C$ is dense.
\end{corollary}

We will now present a result originally due to Isbell \cite{Isbell} and generalised in \cite{Madden} which has no analogue in classical topology.
\begin{lemma}\label{lem:largest_dense_congruence}
 Every $\kappa$-frame $L$ has a largest dense congruence, $\D_L$.
\end{lemma}
\begin{proof}
 Let $\D = \{(a,b) \mid a \wedge x = 0 \iff b \wedge x = 0 \,\text{ for all $x \in L$}\}$.
 It is routine to check that this is an congruence relation on $L$. %
 
 Notice that $\D$ is dense, since if $(0, a) \in \D$ then $0 \wedge a = 0$ implies $a = a \wedge a = 0$.
 Now suppose $C \not\le \D$. We show $C$ is not dense.
 
 Take $(a,b) \in C \setminus \D$ with $a \le b$. There is some $x \in L$ such that $a \wedge x = 0$, but $b \wedge x = c > 0$.
 But then $(0, c) = (a \wedge x, b \wedge x) \in C$, so $\nabla_c \le C$ and $C$ is not dense.
\end{proof}

\begin{corollary}\label{cor:clear_cong_characterisation}
 For every $\kappa$-ideal $I$ on $L$, there is a largest congruence $\partial_I$ dense in $\widetilde{\nabla}_I$.
 Under the correspondence of \cref{rem:lattice_of_quotients}, $\partial_I$ in $\C L$ corresponds to $\D$ in $\C (L/\widetilde{\nabla}_I)$.
 More explicitly, $\partial_I = \{(a,b) \mid a \wedge x \in I \iff b \wedge x \in I \,\text{ for all $x \in L$}\}$.
\end{corollary}

\begin{corollary}\label{cor:clear_cong_heyting_arrow} %
 $\partial_I = \{(a,b) \mid \;\downarrow\!a \rightarrow I \,=\, \;\downarrow\!b \rightarrow I\}$ where `$\rightarrow\!$' denotes the
 Heyting arrow in $\h_\kappa I$.
\end{corollary}

\begin{example}\label{ex:smallest_dense_sublocale_of_reals}
 The space of real numbers $\R$ with the usual topology has both the rationals $\Q$ and the irrationals $\Q^c$ as dense subspaces,
 but these share no points in common. The locale $\Omega(\R)/\D$ is contained in both of these and thus has no points at all.
 But since $\Omega(\R)/\D$ is dense in $\R$, it is certainly not trivial.
\end{example}

The congruences of the form $\partial_I$ are another important class of congruences.
\begin{definition}
 A congruence of the form $\partial_I$ will be called a \emph{clear} congruence.
 If $I$ is a principal ideal $\downarrow\!a$ we sometimes write $\partial_I$ as $\partial_a$.
\end{definition}

The assignment $I \mapsto \partial_I$ is not as well behaved as $\widetilde{\nabla}_\bullet$ and $\Delta_\bullet$. It is injective and reflects order,
but the following example shows it is not monotone.
\begin{example}
 The chain $\mathbbold{3} = \{0, a, 1\}$ is the frame of opens of the Sierpiński space.
 The clear congruences $\partial_0$ and $\partial_a$ correspond to the open and closed points respectively,
 so $\partial_0$ and $\partial_a$ are incomparable.
\end{example}

The clear congruences generate the congruence lattice under meets.
\begin{lemma}\label{lem:meet_of_clear} %
 Every congruence is a meet of clear congruences.
\end{lemma}
\begin{proof}
 Take $C \in \C L$. Certainly, $C \le \bigwedge\{\partial_I \mid \partial_I \ge C,\, I \in \h_\kappa L\}$.
 Now suppose $(a,b) \notin C$. We may assume $a < b$ and we need a $\kappa$-ideal $I$ such that $\partial_I \ge C$ and $(a,b) \notin\partial_I$.
 
 Let $\widetilde{\nabla}_I = \cl(\nabla_a \vee C)$. Then $\nabla_a \vee C$ is dense in $\widetilde{\nabla}_I$ and so $\nabla_a \vee C$ is smaller
 than the largest congruence dense in $\widetilde{\nabla}_I$. Thus, $\partial_I \ge \nabla_a \vee C \ge C$.
 Also, $\nabla_a \le \widetilde{\nabla}_I$ and so $a \in I$.
 
 Suppose $b \in I$ as well. Then $(a,b) \in \widetilde{\nabla}_I \le \nabla_a \vee C$. So $(a,b) = (a\vee a, a\vee b) \in C$.
 This is a contradiction and so $b \notin I$. But then $(a,b) \notin \partial_I$ by \cref{cor:clear_cong_characterisation}.
\end{proof}

Quotients of a $\kappa$-frame by clear congruences have a special form.
\begin{definition}
 A $\kappa$-frame $L$ is called \emph{clear} (or \emph{d-reduced} \cite{Madden}) if $\D_L = 0$ in $\C L$.
\end{definition}
\begin{lemma}\label{lem:quotient_by_clear} %
 The quotient $L/C$ is clear if and only if the congruence $C$ is clear.
\end{lemma}
\begin{proof}
 Simply apply \cref{cor:closure_in_quotient}.
\end{proof}

\begin{lemma}\label{lem:clear_distinct_pseudocomplements} %
 A $\kappa$-frame $L$ is clear if and only if no two distinct principal ideals on $L$ have the same pseudocomplement in $\h_\kappa L$.
\end{lemma}
\begin{proof}
 Apply \cref{cor:clear_cong_heyting_arrow}.
\end{proof}
\begin{corollary}\label{cor:clear_frame_Boolean}%
 A frame is clear if and only if it is Boolean.
\end{corollary}
\begin{proof}
 Suppose $L$ is a clear frame.
 For frames, $(\downarrow\! a)^* \in \h_\infty L$ is again principal and so $\downarrow\! a \vee (\downarrow\! a)^*$ is principal.
 But we also have $(\downarrow\! a \vee (\downarrow\! a)^*)^* = (\downarrow\! a)^* \wedge (\downarrow\! a)^{**} = 0$.
 Obviously, $(\downarrow\! 1)^* = 0$ and so $(\downarrow\! a \vee (\downarrow\! a)^*) = 1$ by \cref{lem:clear_distinct_pseudocomplements}.
 So every principal ideal, and thus every element of $L$, has a complement.
 
 Conversely, suppose $L$ is Boolean. Then every principal ideal of $L$ is complemented and so pseudocomplements are clearly unique.
\end{proof}

We are now in a position to prove the following characterisation of clear $\kappa$-frames.
\begin{theorem}\label{thm:clear_kappa_frames}%
A $\kappa$-frame $L$ is clear if and only if it is isomorphic to a generating sub-$\kappa$-frame of a Boolean frame.
\end{theorem}
\begin{proof}
 Suppose $L$ is clear and consider the map $g: L \to \h_\kappa L/\D_{\h_\kappa L}$ given by the
 composition $L \xrightarrow{\,\downarrow\,} \h_\kappa L \twoheadrightarrow \h_\kappa L/\D_{\h_\kappa L}$.
 The frame $\h_\kappa L/\D_{\h_\kappa L}$ is Boolean by \cref{cor:clear_frame_Boolean} and is generated by the image
 of $L$ under $g$ since $\h_\kappa L$ is generated by the image of $L$ and the quotient map is surjective.
 Finally, we show that $g$ is injective. Take $x,y \in L$. By \cref{cor:clear_cong_heyting_arrow},
 $[\downarrow\! x] = [\downarrow\! y]$ if and only if $(\downarrow\! x)^* = (\downarrow\! y)^*$. But by \cref{lem:clear_distinct_pseudocomplements},
 $(\downarrow\! x)^*$ and $(\downarrow\! y)^*$ are distinct if $x$ and $y$ are. 
 
 Now suppose $L$ is a sub-$\kappa$-frame of a complete Boolean algebra $B$ and that $L$ generates $B$ under arbitrary joins.
 Let $a \in L$ and consider $a^c \in B$. Suppose $b \in L$, $b \le a^c \in B$. Then $b \wedge a = 0$ and so $b \in (\downarrow\! a)^* \in \h_\kappa L$.
 Conversely, if $b \in (\downarrow\! a)^*$ then $b \wedge a = 0$ and so $b \le a^c$. Thus $(\downarrow\! a)^* =\; \downarrow\!a^c \cap L$.
 Since $L$ generates $B$, $a^c = \bigvee (\downarrow\! a)^* \in B$.
 So for any $a, b \in L$, $(\downarrow\! a)^* = (\downarrow\! b)^*$ only if $a^c = b^c$ only if $a = b$. Thus, $L$ is clear.
\end{proof}

\begin{example}\label{ex:clear_kappa_frame} %
 Let $M$ be the lattice of subsets of $\N$ that are either finite or equal to $\N$.
 Since $M$ is a generating sublattice of $2^\N$, it is a clear distributive lattice by \cref{thm:clear_kappa_frames}.
 
 Consider the ideal $I$ of elements of $M$ which represent sets that do not contain the natural number $2$.
 This is the pseudocomplement of the principal ideal generated by $\{2\}$.
 The quotient lattice $M / \widetilde{\nabla}_I$ consists of three elements, namely $[\emptyset]$, $[\{2\}]$ and $[\N]$, and is
 isomorphic to the 3-element frame $\mathbbold{3}$. Thus, quotients of clear $\kappa$-frames may fail to be clear.
\end{example}

\begin{lemma}\label{lem:hereditary_clear_is_Boolean} %
 Every quotient of a $\kappa$-frame $L$ is clear if and only if $L$ is Boolean.
\end{lemma}
\begin{proof}
 $(\Leftarrow)$ A quotient of a Boolean $\kappa$-frame is Boolean and Boolean $\kappa$-frames are clear.
 
 $(\Rightarrow)$
 Let $a \in L$. We will show that $a$ has a complement. Consider $\Delta_a \in \C L$ and
 let $\widetilde{\nabla}_I = \cl(\Delta_a)$. But $\widetilde{\nabla}_I$ is clear so $\widetilde{\nabla}_I = \partial_I$
 and thus $\widetilde{\nabla}_I = \Delta_a$.
 
 Now $\nabla_a \vee \widetilde{\nabla}_I = 1$ and so $(0,1) \in \nabla_a \vee \widetilde{\nabla}_I$. So $(a,1) \in \widetilde{\nabla}_I$
 and $a \vee i = 1$ for some $i \in I$. But then $\nabla_a \vee \nabla_i = 1$ and also $\nabla_a \wedge \nabla_i = 0$ since
 $\nabla_i \le \widetilde{\nabla}_I$. Thus $i$ is the complement of $a$ in $L$ and $L$ is Boolean.
\end{proof}

\subsection{Congruence \texorpdfstring{$\kappa$}{κ}-frames}\label{subsec:congruence_kappa_frames}

By \cref{cor:congruence_lattice_generators,cor:nabla_delta_complements}, the complete lattice of congruences on a $\kappa$-frame $L$
is generated by the principal closed congruences and their complements. \Cref{lem:joins_and_meets_nabla} shows that the principal closed congruences
give a mapping of $L$ into the congruence lattice that preserves finite meets and $\kappa$-joins. Let us consider the freest possible
$\kappa$-frame which satisfies these properties.

\begin{definition} \label{def:congruence_kappa_frame}
 Let $\Cvar L$ be the $\kappa$-frame generated by the symbols $\varnabla_a$ and $\varDelta_a$ for each $a \in L$, subject to
 relations ensuring that the map $a \mapsto \varnabla_a$ is a $\kappa$-frame homomorphism and that
 $\varnabla_a \wedge \varDelta_a = 0$ and $\varnabla_a \vee \varDelta_a = 1$ for all $a \in L$.
\end{definition}

\begin{lemma}\label{lem:nabla_epi}
 The map $\varnabla_\bullet: a \mapsto \varnabla_a$ is an epimorphism.
\end{lemma}
\begin{proof}
 It is clearly a $\kappa$-frame homomorphism. Now suppose $f,g: \Cvar L \to M$ and that $f \varnabla_\bullet = g \varnabla_\bullet$.
 Then $f(\varnabla_a) = g(\varnabla_a)$ for all $a \in L$. Also, $f(\varDelta_a) = f(\varnabla_a)^c = g(\varnabla_a)^c = g(\varDelta_a)$.
 So $f$ and $g$ agree on the generators of $\Cvar L$ and thus they agree everywhere.
\end{proof}

The $\kappa$-frame $\Cvar L$, together with the homomorphism $\varnabla_\bullet$, is readily seen to satisfy a universal property.
\begin{lemma} \label{lem:congruence_frame_universal_prop} %
 If $L$ and $M$ are $\kappa$-frames and $f$ is a $\kappa$-frame homomorphism with image lying in $BM$,
 the Boolean algebra of complemented elements of $M$, then there is a unique $\kappa$-frame homomorphism $\overline{f}$ making the following diagram commute.
 
 \begin{center}
   \begin{tikzpicture}[node distance=3.5cm, auto]
    \node (CL) {$\Cvar L$};
    \node (M) [right of=CL] {$M$};
    \node (L) [below of=CL] {$L$};
    \draw[->] (L) to node {$\varnabla_\bullet$} (CL);
    \draw[->, dashed] (CL) to node {$\overline{f}$} (M);
    \draw[->] (L) to node [swap, yshift=-2pt, xshift=-7pt] {$f$} (M);
   \end{tikzpicture}
 \end{center}
 
 Furthermore, $\varnabla_\bullet: L \to \Cvar L$ is essentially unique amongst all morphisms $g:L \to N$ which satisfy the universal property above
 and for which the image of $g$ lies in $BN$.
\end{lemma}
\begin{proof}
 We can define a map $\widetilde{f}$ from the free frame on $L$ generated by $\varnabla_a$ and $\varDelta_a$ by $\widetilde{f}(\varnabla_a) = f(a)$ and
 $\widetilde{f}(\varDelta_a) = f(a)^c$. This factors through $\Cvar L$ by \cref{rem:factor_through_quotients} to give the desired map $\overline{f}$.
 Uniqueness is guaranteed since $\varnabla_\bullet$ is an epimorphism.
 
 Now suppose $g: L \to N$ satisfies the conditions above. We apply the universal properties of $\varnabla_\bullet$ and $g$ in turn to give the diagram below.
 
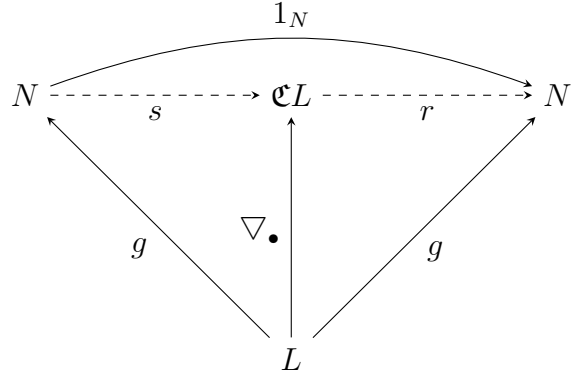
\begin{center}
   \begin{tikzpicture}[node distance=3.5cm, auto]
    \node (CL) {$\Cvar L$};
    \node (N) [left of=CL] {$N$};
    \node (M) [right of=CL] {$N$};
    \node (L) [below of=CL] {$L$};
    \draw[->] (L) to node {$\varnabla_\bullet$} (CL);
    \draw[->, dashed] (N) to node [swap] {$s$} (CL);
    \draw[->, dashed] (CL) to node [swap] {$r$} (M);
    \draw[->, bend left=20] (N) to node {$1_N$} (M);
    \draw[->] (L) to node [swap, yshift=-1pt, xshift=-3pt] {$g$} (M);
    \draw[->] (L) to node {$g$} (N);
   \end{tikzpicture}
 \end{center}
 Composing the induced maps $N \to \Cvar L$ and $\Cvar L \to N$ gives a map $N \to N$ making the outer triangle commute, which must therefore be the
 identity by uniqueness. Hence, $rs = 1_N$. Similarly, composing in the opposite direction gives $sr = 1_{\Cvar L}$. So $r: \Cvar L \cong N$ and
 $g = r \circ \varnabla_\bullet$, which proves the result.
\end{proof}

\begin{lemma}\label{lem:universal_map_into_complemented_is_injective}
 The map $\varnabla_\bullet$ is a monomorphism.
\end{lemma}
\begin{proof}
 Take $a, b\in L$ with $a < b$. Then $(0,a) \in \partial_a$, but $(0,b) \notin \partial_a$ and so $q: L \twoheadrightarrow L/\partial_a$ distinguishes
 $a$ and $b$. Now by \cref{thm:clear_kappa_frames}, we can embed $L/\partial_a$ in a Boolean frame $B$.
 The composite morphism $f: L \to B$ then separates $a$ and $b$. Since $B$ is Boolean, $f$ factors through $\varnabla_\bullet: L \to \Cvar L$ by
 the universal property and so $\varnabla_\bullet$ also distinguishes $a$ and $b$. Thus, $\varnabla_\bullet$ is injective.
\end{proof}

\begin{remark}
The universal property of $\varnabla_\bullet$ turns $\Cvar$ into a functor and $\varnabla_\bullet$ into a natural transformation
$\varnabla: \mathbbold{1} \to \Cvar$. If $f: L \to M$ then we obtain $\Cvar f$ from the universal property applied to $\varnabla_\bullet f$
as shown in the diagram.
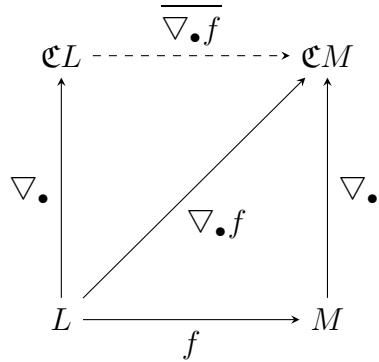
\begin{center}
   \begin{tikzpicture}[node distance=3.5cm, auto]
    \node (CL) {$\Cvar L$};
    \node (CM) [right of=CL] {$\Cvar M$};
    \node (L) [below of=CL] {$L$};
    \node (M) [right of=L] {$M$};
    \draw[->] (L) to node {$\varnabla_\bullet$} (CL);
    \draw[->] (M) to node [swap] {$\varnabla_\bullet$} (CM);
    \draw[->] (L) to node [swap] {$f$} (M);
    \draw[->] (L) to node [swap, yshift=-2pt, xshift=-7pt] {$\varnabla_\bullet f$} (CM);
    \draw[->, dashed] (CL) to node {$\overline{\varnabla_\bullet f}$} (CM);
   \end{tikzpicture}
\end{center}

Functoriality follows easily by composing two such diagrams for $f:L \to M$ and $g:M \to N$ and using commutativity together with uniqueness of the
map $\overline{\varnabla_\bullet f g}$.
\begin{center}
   \begin{tikzpicture}[node distance=3.5cm, auto]
    \node (CL) {$\Cvar L$};
    \node (CM) [right of=CL] {$\Cvar M$};
    \node (L) [below of=CL] {$L$};
    \node (M) [right of=L] {$M$};
    \node (N) [right of=M] {$N$};
    \node (CN) [right of=CM] {$\Cvar N$};
    \draw[->] (L) to node {$\varnabla_\bullet$} (CL);
    \draw[->] (M) to node {$\varnabla_\bullet$} (CM);
    \draw[->] (N) to node [swap] {$\varnabla_\bullet$} (CN);
    \draw[->] (L) to node [swap] {$f$} (M);
    \draw[->] (M) to node [swap] {$g$} (N);
    \draw[->, dashed] (CL) to node {$\overline{\varnabla_\bullet f}$} (CM);
    \draw[->, dashed] (CM) to node {$\overline{\varnabla_\bullet g}$} (CN);
   \end{tikzpicture}
\end{center}
\end{remark}

The $\kappa$-frame $\Cvar L$ is in fact strongly related to the lattice of $\kappa$-frame congruences of $L$.
\Cref{rem:lattice_of_quotients} tells us that if $C$ is a congruence on $L$, congruences on $L/C$ correspond bijectively
to those congruences on $L$ lying above $C$. Once we show $\C L$ is a frame, this will mean that congruences on $L$ correspond
to closed congruences on $\C L$ by \cref{lem:open_closed_lowerset_and_upperset}. With this idea in mind, we now relate congruences
on $L$ to generalised closed congruences on $\Cvar L$. The proof follows that of Madden in \cite{Madden}.

\begin{theorem}\label{thm:quotients_are_closed_quotients_of_congruence_frame}
 There is an order isomorphism between the lattice of congruences on $L$ and the lattice of generalised closed congruences on $\Cvar L$.
\end{theorem}
\begin{proof}
 We wish to lift congruences on $L$ to congruences on $\Cvar L$. A natural way to do this is to map the congruence $C$ on $L$ to
 the congruence $\overline{C}$ on $\Cvar L$ generated by the set $\{(\varnabla_a, \varnabla_b) \mid (a,b) \in C\}$.
 Since $\varnabla_a \sim \varnabla_b$ if and only if $0 \sim \varnabla_b \wedge \varDelta_a$, we find that
 $\overline{C} = \langle (0, \varnabla_b \wedge \varDelta_a) \mid (a,b) \in C\rangle$, which is clearly generalised closed.
 
 There is another way we might lift congruences on $L$ to $\Cvar L$. If $C \in \C L$, we may apply $\Cvar$ to the quotient map $q: L \twoheadrightarrow L/C$.
 This gives a surjection $\Cvar q : \Cvar L \to \Cvar (L/C)$ and we denote the corresponding congruence by $\Cvar(C)$.
 In fact, we will show that $\Cvar(C)$ and $\overline{C}$ coincide in \cref{cor:kernel_of_Cf}, but we only need a weaker result here. Observe that
 $(\varnabla_a, \varnabla_b) \in \Cvar(C) \iff \Cvar q (\varnabla_a) = \Cvar q (\varnabla_b) \iff \varnabla_{q(a)} = \varnabla_{q(b)}
 \iff q(a) = q(b) \iff (a,b) \in C$ (using \cref{lem:universal_map_into_complemented_is_injective} for the third equivalence).
 So $\overline{C} \le \Cvar(C)$ and $(\varnabla_a, \varnabla_b) \in \overline{C} \iff (a,b) \in C$. Thus the map $C \mapsto \overline{C}$ is injective.
 
 On the other hand, $\Cvar L$ is generated by elements of the form $\varDelta_a \wedge \varnabla_b$ and so every generalised closed congruence on $\Cvar L$
 is generated by pairs of the form $(0, \varDelta_a \wedge \varnabla_b)$ and thus by pairs of the form $(\varnabla_a, \varnabla_b)$. Such a generalised closed
 congruence corresponds to the congruence on $L$ generated by the corresponding pairs of the form $(a,b)$.
 
 So the map $C \mapsto \overline{C}$ is bijective. It also clearly preserves joins and thus it gives an order isomorphism from $\C L$ to $\h_\kappa \Cvar L$.
\end{proof}
\begin{corollary}\label{cor:kappa_frame_congruences_form_a_frame}
 The lattice of congruences on a $\kappa$-frame is a frame. Furthermore, $\C$ can be made into a functor with $\C \cong \h_\kappa \circ \Cvar$.
\end{corollary}
\begin{corollary}\label{cor:congruence_kappa_frame_Lindelof}
 The congruence frame of a $\kappa$-frame is $\kappa$-Lindelöf.
\end{corollary}

\begin{remark}
 The $\kappa$-frame $\Cvar L$ can now be viewed as a sub-$\kappa$-frame of $\C L$. Elements of $\Cvar L$ correspond
 to congruences on $L$ generated by fewer than $\kappa$ generators. For this reason, $\Cvar L$ is called the \emph{congruence $\kappa$-frame} of $L$.
 Since $\varnabla_\bullet$ and $\varDelta_\bullet$ are simply restrictions of $\nabla_\bullet$ and $\Delta_\bullet$, we will henceforth
 represent them by the same symbols.
\end{remark}

Notice that the lift congruence $\overline{C}$ from \cref{thm:quotients_are_closed_quotients_of_congruence_frame}
may be expressed as the join of principal closed congruences $\nabla_A$ for each congruence $A \le C$ with fewer than $\kappa$ generators.
In the case that $L$ is a frame, $\overline{C} = \nabla_C$. We will abuse notation and write $\overline{C} = \widetilde{\nabla}_C$ in general.

We recap the properties of the congruence lattice in light of the fact that it is a frame.
\begin{proposition}\label{prop:congruence_frame_summary}
 The congruence lattice of a $\kappa$-frame $L$ is a zero-dimensional frame $\C L$ and the map
 $\widetilde{\nabla}:\h_\kappa L \to \C L$ is an injective frame homomorphism. Furthermore, if $C \in \C L$
 then $\C (L/C)$ is canonically isomorphic to $\C L / \nabla_C$.
\end{proposition}

We would like to be able to present a similar result for congruence $\kappa$-frames, but the last claim is not immediate.
We will need to use a result from \cref{subsec:lindelof_congruences} relating the congruences on a $\kappa$-frame $L$ and the congruences on the
frame of $\kappa$-ideals on $L$.
\begin{proposition}\label{prop:congruence_kappa_frame_summary} %
 The lattice of the congruences on a $\kappa$-frame $L$ that have fewer than $\kappa$ generators is a zero-dimensional $\kappa$-frame $\Cvar L$
 and the map $\nabla_\bullet: L \to \Cvar L$ is an injective $\kappa$-frame homomorphism. Furthermore, if $C \in \Cvar L$
 then $\Cvar (L/C)$ is canonically isomorphic to $\Cvar L / \widetilde{\nabla}_C$.
\end{proposition}
\begin{proof}
 For the final claim, note that $\h_\kappa\Cvar (L/C) \cong \h_\kappa\Cvar L / \nabla_C \cong \h_\kappa (\Cvar L / \widetilde{\nabla}_C)$
 by \cref{prop:congruence_frame_summary,lem:congruences_on_frames_of_ideals}. Thus we obtain $\Cvar (L/C) \cong \Cvar L / \widetilde{\nabla}_C$ as required. 
\end{proof}

We now examine how congruences lift from $L$ to $\Cvar L$ in more detail.

\begin{lemma}\label{lem:kernel_on_congruence_frame} %
 If $L$ and $M$ are $\kappa$-frames and $\varphi: \Cvar L \to M$ is a homomorphism, then
 $\cl(\ker \varphi) = \widetilde{\nabla}_{\ker (\varphi \circ \nabla_\bullet)}$.
\end{lemma}
\begin{proof}
 Observe the following sequence of equivalences.
 \begin{align*}
  (0, \Delta_a \wedge \nabla_b) \in \ker \varphi
  &\iff (\nabla_a, \nabla_b) \in \ker \varphi \\
  &\iff (a,b) \in \ker(\varphi \circ \nabla_\bullet) \\
  &\iff (\nabla_a, \nabla_b) \in \widetilde{\nabla}_{\ker(\varphi\circ\nabla_\bullet)} \\
  &\iff (0, \Delta_a \wedge \nabla_b) \in \widetilde{\nabla}_{\ker(\varphi\circ\nabla_\bullet)}%
 \end{align*}
 Since congruences of the form $\Delta_a \wedge \nabla_b$ generate $\Cvar L$, the result follows.
\end{proof}
\begin{corollary}\label{cor:kernel_of_Cf}
 If $f: L \to M$ is a $\kappa$-frame homomorphism then $\cl(\ker(\Cvar f)) = \widetilde{\nabla}_{\ker f}$.
 Furthermore, $\ker(\Cvar f) = \widetilde{\nabla}_{\ker f}$ if $f$ is surjective.
\end{corollary}
\begin{proof}
 For the first part, simply recall that $\nabla_\bullet \circ f  = \Cvar f \circ \nabla_\bullet$ from the definition of $\Cvar f$ and
 then notice that $\ker(\nabla_\bullet \circ f) = \ker f$ since $\nabla_\bullet$ is injective.
 
 Now suppose $f: L \twoheadrightarrow L/C$ is a quotient map.
 By \cref{prop:congruence_kappa_frame_summary}, $\Cvar(L/C) \cong \Cvar L / \widetilde{\nabla}_C$.
The composition of this isomorphism with $\Cvar f$ gives a map from $\Cvar L$ to $\Cvar L / \widetilde{\nabla}_C$ whose kernel
is readily seen to be $\widetilde{\nabla}_C$.
\end{proof}
\begin{example}
 The above surjectivity requirement is necessary.
 Let $M$ be the complete atomic Boolean algebra with $|\R|$ atoms, $L$ the frame of opens of $\R$ and $f: L \hookrightarrow M$ the usual inclusion map.
 The above corollary shows that $\C f$ is dense. However, it fails to be injective.
 Every subframe of the frame $\C M \cong M$ is spatial (see \cref{lem:congruence_frame_of_Boolean_frame}), but in \cref{subsec:spatial_reflection}
 we will show that $\C L$ is spatial if and only if every quotient of $L$ is spatial. Since $L/\D$ is a non-spatial quotient
 of $L$, $\C L$ is non-spatial and so $f$ cannot be injective.
\end{example}

\subsection{The congruence tower}\label{subsec:congruence_tower}

We constructed $\Cvar L$ by freely adjoining complements to a $\kappa$-frame $L$.
We might think of the functor $\Cvar$ as making $L$ `more Boolean', but $\Cvar L$ is seldom Boolean itself.
Repeated application of $\Cvar$ (taking colimits at limit ordinals) leads to a
transfinite sequence $\Cvar L, \Cvar^2 L, \Cvar^3 L, \dots$ which is known as the \emph{congruence tower} of $L$.

The next lemma shows that if a Boolean $\kappa$-frame is reached, the sequence stabilises.
\begin{lemma}\label{lem:congruence_frame_of_Boolean_frame}
 $\nabla_\bullet: L \to \Cvar L$ is an isomorphism if and only if $L$ is Boolean.
\end{lemma}
\begin{proof}
 If $L$ is Boolean then the identity map on $L$ factors through $\nabla_\bullet$. So $\nabla_\bullet$ is a split monomorphism as well as an epimorphism and
 thus an isomorphism.
 
 If $\nabla_\bullet: L \to \Cvar L$ is an isomorphism then every $a \in L$ has a complement $(\nabla_\bullet)^{-1}(\Delta_a)$ and so $L$ is Boolean.
\end{proof}

In the case of $\kappa$-frames (\emph{but not frames!}), Madden showed in \cite{Madden} that the congruence tower eventually yields a Boolean $\kappa$-frame
after at most $\kappa$-many operations. The resulting $\kappa$-frame $\mathfrak{B} L$ is the \emph{Boolean reflection} of $L$ and any $\kappa$-frame
homomorphism from $L$ to a Boolean $\kappa$-frame factors uniquely through the canonical map $d:L \to \mathfrak{B} L$.

Such a Boolean reflection might fail to exist for frames. If $L$ is the free frame on $\N$ then $\mathfrak{B} L$ would be the free complete Boolean
algebra on $\N$, but Hales \cite{Hales} and Gaifman \cite{Gaifman} showed that no such object exists. Thus, the transfinite sequence of iterated
congruence frames of a frame may never stabilise.

The congruence tower can also be used to characterise epimorphisms of frames and $\kappa$-frames.
A $\kappa$-frame homomorphism $f: L \to M$ is an epimorphism if and only if $\mathfrak{B} f$ is a surjection.
A frame homomorphism $f: L \to M$ is an epimorphism if and only if there is an ordinal $\alpha$ such that the image of $\C^\alpha f$
in $\C^\alpha M$ contains the image of the canonical map $\nabla^\alpha$ from $M$ to $\Cvar^\alpha M$.
These characterisations were shown in \cite{Madden} and \cite{EpimorphismFrm} respectively.

\subsection{Clear congruences on frames}\label{subsec:clear_congruences_on_frames}

In \cref{subsec:boolean_congruence_frames}, we will answer the question of which frames have a Boolean congruence frame,
but before we can do that we will need to know some more about clear congruences on frames.

Clear congruences on frames are slightly better behaved than clear congruences on $\kappa$-frames.
Since frames are clear if and only if they are Boolean, the correspondence between quotients and congruences
means that if $L$ is a frame, any congruence lying above a clear congruence in $\C L$ is itself clear.
We will restrict our consideration to frame congruences for the rest of the section.

When $L$ is a frame, the the largest dense congruence $\D$ corresponds to the nucleus $a \mapsto a^{**}$.
Elements of the form $a^{**}$ are called \emph{regular elements} of $L$ and $L/\D$ can
be thought of as the complete Boolean algebra of regular elements of $L$.
We might ask which congruences lie above $\D$.

The following lemma gives an easy way to compute $\D \vee C$
for any congruence $C \in \C L$. One surprising consequence is that $\D \vee C$ only depends on the closure of $C$. This was first shown
by Beazer and Macnab in \cite{Macnab}.

\begin{lemma}\label{lem:join_with_largest_dense_cong} %
 If $C \in \C L$ and $\cl(C) = \nabla_a$ then $\D \vee C = \partial_{a^{**}}$.
\end{lemma}
\begin{proof}
 By \cref{cor:closed_congruences_in_quotient}, $\cl(\D \vee \nabla_a) = \nabla_{a^{**}}$
 and so $\D \vee \nabla_a = \partial_{a^{**}}$ since $\D \vee \nabla_a$ is clear.
 We thus have $\D \le \partial_{a^{**}} \le \D \vee C \le \D \vee \partial_a$ and we need only show $\partial_a \le \partial_{a^{**}}$.
 
 Let $(x,y) \in \partial_a$. We may assume $x \le y$ without loss of generality.
 Now suppose $x \wedge z \le a^{**}$ for some $z \in L$. Then $x \wedge z \wedge a^* \le 0 \le a$ and so $y \wedge z \wedge a^* \le a$
 since $(x,y) \in \partial_a$. But then $y \wedge z \wedge a^* \le a \wedge a^* = 0$ and so $y \wedge z \le a^{**}$.
 Thus $(x,y) \in \partial_{a^{**}}$ as required.
\end{proof}
\begin{remark}
 In a similar way one can show that $\partial_b \vee \nabla_a = \partial_{(a \rightarrow b) \rightarrow b}$ whenever $b \le a$.
\end{remark}
\begin{remark}
 Notice that while the map $\partial_\bullet$ does not even preserve order, the map $a \mapsto \partial_{a^{**}}$ is actually a frame homomorphism
 from $L$ to $\uparrow\! \D \subseteq \C L$.
\end{remark}
\begin{corollary}\label{cor:congruences_above_largest_dense_cong} %
 The congruences lying above $\D$ form a frame isomorphic to $L / \D$.
\end{corollary}

\subsection{Boolean congruence frames}\label{subsec:boolean_congruence_frames}
We can now tackle the characterisation of the frames $L$ for which $\C L$ is Boolean.
Such frames are called \emph{scattered} \cite{PleweRareSublocales}.

\begin{lemma}\label{lem:no_dense_elements_Boolean}
 A frame $M$ is Boolean if and only if it has no nontrivial dense elements.
\end{lemma}
\begin{proof}
 The forward implication is trivial.
 Suppose $M$ has no nontrivial dense elements and consider an element $a \in M$. Certainly, $a \wedge a^* = 0$.
 But $a \vee a^*$ is a dense element of $M$, since $(a \vee a^*)^* = a^* \wedge a^{**} = 0$.
 But then $a \vee a^* = 1$ and so $a$ is complemented.
\end{proof}

\begin{remark}
A dense congruence is \emph{not} the same as a dense element of $\C L$. Rather,
dense elements of $\C L$ are called \emph{rare} congruences. They were first discussed in \cite{PleweRareSublocales}
using localic terminology. \Cref{lem:no_dense_elements_Boolean} implies that $\C L$ is Boolean if and only if $L$ has no nontrivial rare congruences.
\end{remark}

\begin{lemma}\label{lem:rare_congruence} %
 A congruence $C \in \C L$ is rare if and only if whenever $a, b \in L$ and $a < b$, there exists a pair $(c, d) \in C$ with $a \le c < d \le b$.
\end{lemma}
\begin{proof}
A congruence $C$ is rare if and only if $C$ meets every nonzero element of $\C L$. Since principal congruences generate $\C L$, it is enough
to consider meets with these and so $C$ is rare if and only if $C \wedge \langle(a,b)\rangle \ne 0$ whenever $a < b$.
Hence $C$ is rare if whenever $a < b$ there are $x, y \in L$, $x < y$ such that $(x,y) \in C \cap \langle(a,b)\rangle$.

($\Leftarrow$) If $(c,d) \in C$ with $a \le c < d \le b$ then certainly $(c,d) \in C \cap \langle(a,b)\rangle$.

($\Rightarrow$) Suppose we have $(x,y) \in C \cap \langle(a,b)\rangle$ with $x < y$. Then $(x,y) \notin \langle(a,b)\rangle^c = \nabla_a \vee \Delta_b$.
So $(x \vee a) \wedge b \ne (y \vee a) \wedge b$. So letting $c = (x \vee a) \wedge b$ and $d = (y \vee a) \wedge b$, we have $a \le c < d \le b$ and
$(c,d) \in C$ as required.
\end{proof}

As an application of the above characterisation, we may prove the following.
\begin{corollary} %
Suppose $L$ is a complete chain. Then $\C L$ is Boolean if and only if $L$ is well-ordered.
\end{corollary}
\begin{proof}
By \cref{lem:no_dense_elements_Boolean}, we may instead show there are no nontrivial rare congruences on $L$ if and only if $L$ is well-ordered.

($\Leftarrow$) Suppose $L$ is well-ordered and let $C$ be a rare congruence on $L$. Suppose $C \ne 1$. Let $x = \bigvee [0]_C < 1$.
Then the set $S = \{y \in L \mid y > x\}$ is non-empty and therefore has a least element $s$. Since $C$ is rare, there exists a pair $(a,b) \in C$
such that $x \le a < b \le s$. Since $a < s$, we have $a \notin S$ and so $a = x$. Then $(0,a) \in C$ and $(a,b) \in C$ means $(0,b) \in C$.
But $b > x$. So this is a contradiction and thus $C = 1$.

($\Rightarrow$) Conversely, suppose $S \subseteq L$ is a nonempty set without a least element and let $y = \bigwedge S$.
Then we claim that $\partial_y = \{(a,b) \mid a \le y \iff b \le y\}$ is rare. We need only show that whenever $y < x$ there is a $z \in L$
such that $y < z < x$. But since $x > y$, there is a $z \in S$ with $z < x$ and certainly $y < z$.
\end{proof}

The next lemma, taken from \cite{PleweRareSublocales}, shows that we only need to consider the clear congruences in order to show $\C L$ is Boolean.
\begin{lemma} %
 $\C L$ is Boolean if and only if every clear congruence is complemented.
\end{lemma}
\begin{proof}
 Let $C$ be a rare congruence and let $\nabla_a = \cl(C)$. Then $\nabla_a \le C \le \partial_a$. So $\partial_a$ is also rare.
 But $\partial_a$ rare and complemented means $\partial_a = 1$. So $a = 1$ and therefore $C = 1$. Thus, $\C L$ is Boolean by
 \cref{lem:no_dense_elements_Boolean}.
\end{proof} %

The final piece of the puzzle is to characterise when the clear congruences are complemented.
This was shown by Beazer and Macnab in \cite{Macnab}.
\begin{theorem}\label{thm:complemented_clear_congruence}
 The congruence $\partial_a$ is complemented if and only if there is a least $b \in L$ with $b \ge a$ and $(b,1) \in \partial_a$.
\end{theorem}
\begin{proof}
 ($\Rightarrow$) Suppose $\partial_a$ has a complement $C$. Moving to the quotient $L/\nabla_a$, we get that $\D_{L/\nabla_a}$ and
 $\overline{C} = [C]_{L/\nabla_a}$ are complements. Now by \cref{lem:join_with_largest_dense_cong},
 $\D_{L/\nabla_a} \vee \overline{C} = \D_{L/\nabla_a} \vee \cl(\overline{C})$.
 But $\overline{C}$ is the smallest congruence $B$ for which $\D_{L/\nabla_a} \vee B = 1$, so $\overline{C} = \cl(\overline{C})$.
 Thus $\D_{L/\nabla_a}$ is an open congruence and so $[1]_{\D_{L/\nabla_a}}$ has a least element $[b] \in L/\nabla_a$.
 Then $a \vee b$ is a well-defined element of $L$ satisfying our requirements.
 
 ($\Leftarrow$) Suppose there is a smallest $b \ge a$ such that $(b,1) \in \partial_a$. We show $\partial_a = \nabla_a \vee \Delta_b$.
 We always have $\nabla_a \le \partial_a$ and $\Delta_b \le \partial_a$ since $(b,1) \in \partial_a$.
 Now let $(x,y) \in \partial_a$ with $x \le y$. We need $(x \vee a) \wedge b = (y \vee a) \wedge b$.
 
 Choose $z$ as large as possible so that $x \wedge z \le a$ and consider $x \vee z$.
 Let $w \in L$ be such that $(x \vee z) \wedge w \le a$. Then $(x \wedge w) \vee (z \wedge w) \le a$ and
 so $x \wedge w \le a$ and $z \wedge w \le a$. From the former, $w \le z$, and then from the latter,
 $w = z \wedge w \le a$. Thus, $(x \vee z, 1) \in \partial_a$ and so $b \le x \vee z$.
 
 Let $\overline{x} = x \vee a$ and $\overline{y} = y \vee a$. Notice that and $y \wedge z \le a$ since $(x,y) \in \partial_a$ and that $z \ge a$.
 So $a \le z \wedge \overline{y} \le a$ and $x \le x \wedge \overline{y} \le x \vee a$. Now
 $\overline{x} = (x \wedge \overline{y}) \vee (z \wedge \overline{y}) = (x \vee z) \wedge \overline{y}$.
 So we have $b \wedge \overline{x}= b \wedge (x \vee z) \wedge \overline{y} = b \wedge \overline{y}$, as required.
\end{proof}
\begin{corollary}
 Every complemented clear congruence is a join of an open congruence and a closed congruence.
\end{corollary}
\begin{corollary}
 The congruence frame of a frame $L$ is Boolean if and only if for every $a \in L$, there is a smallest $b \in L$ with $b \ge a$ and $(b,1) \in \partial_a$.
\end{corollary}

The characterisations of the frames $L$ for which $\C^2 L$, $\C^3 L$ and $\C^4 L$ are Boolean can be found in \cite{PleweRareSublocales}.

\section{Strictly zero-dimensional biframes}\label{section:strictly_zero_dimensional_biframes} %

\subsection{Congruence biframes}

Congruence frames have a natural biframe structure \cite{FrithCong}.
\begin{definition}
 Suppose $L$ is a frame or $\kappa$-frame. Let $\nabla L$ be the subframe of $\C L$ consisting of the generalised closed congruences and let $\Delta L$
 be the subframe of $\C L$ generated by the principal open congruences. Then $\C L = (\C L, \nabla L, \Delta L)$ is called the
 \emph{congruence biframe} of $L$.
\end{definition}

Notice that if $f: L \to M$ is a $\kappa$-frame homomorphism, then $\C f$ maps principal closed congruences to principal closed congruences and
principal open congruences to principal open congruences. Thus $\C f$ is a biframe homomorphism and the congruence functor can be viewed as
mapping into the category of biframes.

The congruence biframe is zero-dimensional: $\nabla L$ is generated by principal closed congruences, which have complements in $\Delta L$, and
$\Delta L$ is generated by the principal open congruences, which have complements in $\nabla L$.
If we restrict to the case of frames, we can say even more.
\begin{definition}
 A biframe $L = (L_0, L_1, L_2)$ is \emph{strictly zero-dimensional} if every element $a \in L_1$ is complemented in $L_0$ with its complement lying in $L_2$
 and furthermore these complements generate $L_2$. We will call $L$ a \emph{strictly zero-dimensional biframe over $L_1$}.
\end{definition}
\begin{remark}
 The original definition in \cite{StrictlyZeroDimensional} allowed the roles of $L_1$ and $L_2$ above to be interchanged, but as has become standard,
 we will fix the chirality.
\end{remark}
The congruence biframe of a frame is strictly zero-dimensional. We may view the congruence functor as mapping from the category of frames
to the category of strictly zero-dimensional biframes (and biframe homomorphisms), $\StrZdBiFrm$.
The resulting functor $\C:\Frm \to \StrZdBiFrm$ is then fully faithful, since any map from a strictly zero-dimensional biframe is
determined by its action on the first part. Thus $\C$ gives a (slightly non-standard) embedding of the category of frames into the category of biframes.

The current section is largely motivated by the following rather satisfactory characterisation of $\C$.
\begin{proposition}\label{prop:congruence_free_str0dbifrm} %
 Let $\P: \StrZdBiFrm \to \Frm$ be the functor that takes the first part of a strictly zero-dimensional biframe.
 Then $\C: \Frm \to \StrZdBiFrm$ is left adjoint to $\P$.
\end{proposition}
\begin{proof}
 This fact is simply a re-imagining of the usual universal property of \cref{lem:congruence_frame_universal_prop}.
 Let $L$ be a frame and $M$ a strictly zero-dimensional biframe. Suppose we have a frame homomorphism $f: L \to \P M$.
 Then every element in the image of $f$ has a complement in $M_0$ and thus there is a unique $\overline{f}: \C L \to M$ making the diagram below commute.
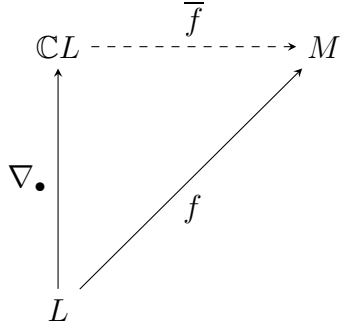
\begin{center}
   \begin{tikzpicture}[node distance=3.5cm, auto]
    \node (CL) {$\C L$};
    \node (M) [right of=CL] {$M$};
    \node (L) [below of=CL] {$L$};
    \draw[->] (L) to node {$\nabla_\bullet$} (CL);
    \draw[->, dashed] (CL) to node {$\overline{f}$} (M);
    \draw[->] (L) to node [swap, yshift=-2pt, xshift=-7pt] {$f$} (M);
   \end{tikzpicture}
\end{center}
 The commutativity of the diagram ensures that $\overline{f}$ maps $\nabla L$ into $M_1$. Then $\overline{f}$ maps $\Delta L$ into $M_2$ since these
 subframes are both generated by the complements of elements in the first parts. So $\overline{f}$ is a biframe homomorphism and $\C$ satisfies the universal
 property of the left adjoint of $\P$.
\end{proof}
\begin{corollary} %
 The category of congruence biframes is a coreflective subcategory of the category $\StrZdBiFrm$.
 We will call this the \emph{congruential coreflection} of a strictly zero-dimensional biframe.
\end{corollary}

\begin{remark}
 In \cref{subsec:congruence_kappa_frames}, the congruence frame was described as the frame obtained by freely adjoining complements to every element in a
 given frame. In this spirit, strictly zero-dimensional biframes are obtained from their first parts by adjoining complements in a potentially more general way.
 \Cref{prop:congruence_free_str0dbifrm} then justifies our previous claim by showing that \emph{$\C L$ is the free strictly zero-dimensional biframe over $L$}.
\end{remark}

\subsection{Characterisations of strictly zero-dimensional biframes}\label{subsec:characterisations_of_str0d_biframes}

We may use the congruence biframe to completely characterise strictly zero-dimensional biframes over a given frame $L$.
It is perhaps surprising that the total parts of the strictly zero-dimensional biframes over $L$ depend only on the total part of $\C L$.
\begin{theorem}\label{thm:dense_quotients_of_congruence_frame} %
 The strictly zero-dimensional biframes over $L$ are precisely the dense quotients of $\C L$.
\end{theorem}
\begin{proof}
 Let $M$ be a strictly zero-dimensional biframe over $L$ and let  $\chi: \C L \to M$ be its congruential coreflection.
 Consider the following diagram.
 
 
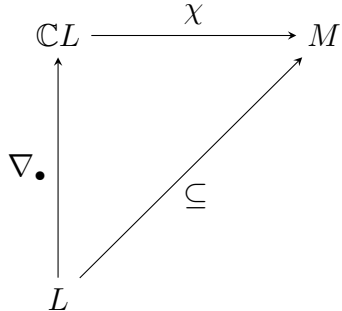
\begin{center}
   \begin{tikzpicture}[node distance=3.5cm, auto]
    \node (CL) {$\C L$};
    \node (M) [right of=CL] {$M$};
    \node (L) [below of=CL] {$L$};
    \draw[->] (L) to node {$\nabla_\bullet$} (CL);
    \draw[->] (CL) to node {$\chi$} (M);
    \draw[->] (L) to node [swap, yshift=-2pt, xshift=-7pt] {$\subseteq$} (M);
   \end{tikzpicture}
 \end{center}
 
 By commutativity of the diagram, every element $a \in L \subseteq M$ is $\chi(\nabla_a)$. Since $\nabla_a$ has a complement in $\C L$,
 every element $a^c \in M$ for $a \in L$ is also in the image of $\chi$. But such elements generate $M_2$
 since $M$ is strictly zero-dimensional. Thus $\chi_1$, $\chi_2$ and $\chi_0$ are surjective.
 
 Now by \cref{lem:kernel_on_congruence_frame}, $\cl(\ker \chi) = \nabla_{\ker(\chi \circ \nabla_\bullet)} = \nabla_{\ker(\subseteq)} = 0$
 and so $\chi$ is dense.
 
 Conversely, if $\chi: \C L \to M$ is a dense quotient, $\P \chi$ is injective by the same reasoning and so $M_1 \cong L$.
 The complements of elements of $\nabla L$ map to complements of elements of $M_1$ under $\chi$ and these generate $M_2$ since the former
 generate $\Delta L$. Thus, $M$ is strictly zero-dimensional.
\end{proof}
\begin{corollary}\label{cor:frame_of_str0d_biframes_over_L} %
 The poset of (isomorphism classes of) strictly zero-dimensional biframes over $L$, ordered in the natural way, is dually isomorphic to
 the frame $\C^2 L / \Delta_{\D_{\C L}}$. The congruence frame $\C L$ is the largest strictly zero-dimensional biframe over $L$, while
 $\C L / \D_{\C L}$ is the smallest.
\end{corollary}

\begin{definition}
 A strictly zero-dimensional biframe is called \emph{discrete} if its total part is Boolean.
\end{definition}
\begin{corollary}\label{cor:str0d_biframe_boolean_total_part} %
 A strictly zero-dimensional biframe over $L$ is discrete if and only if it is isomorphic to $\C L / \D_{\C L}$.
\end{corollary}
\begin{proof}
 First notice that $\C L / \D_{\C L}$ is Boolean since $\D_{\C L}$ is clear. Now suppose $\C L / C$ is a strictly zero-dimensional biframe over $L$
 with Boolean total part. Then $C = \partial_c$ for some $c \in \C L$. But $\partial_c \le \D$ then implies $c = 0$ since $\partial_\bullet$ reflects order.
\end{proof}
\begin{corollary}\label{cor:scattered_frames_have_unique_str0d_biframe_structure} %
 There is a unique strictly zero-dimensional biframe over $L$ if and only if $\C L$ is a Boolean frame (i.e.\ if $L$ is scattered).
\end{corollary}

If the congruence biframe of $L$ is the `free strictly zero-dimensional biframe over $L$', then a `free strictly zero-dimensional biframe'
would be the congruence biframe of a free frame. As one might hope, every strictly zero-dimensional biframe is a quotient of a free strictly
zero-dimensional biframe.
\begin{lemma}\label{lem:quotient_of_free_str0dbifrm} %
 Every strictly zero-dimensional biframe is a quotient of a free strictly zero-dimensional biframe.
 Congruence biframes are precisely the closed quotients of free strictly zero-dimensional biframes.
\end{lemma}
\begin{proof}
 Let $L$ be a frame. Then $L$ is a quotient of a free frame, $L \cong F/C$. Applying the functor $\C$ we find $\C L \cong \C F / \nabla_C$
 and so $\C L$ is a closed quotient of $\C F$. The converse holds since $\C F / \nabla_C \cong \C (F / C)$ for any closed congruence $\nabla_C$.
 
 Now let $M$ be a strictly zero-dimensional biframe. Then $M$ is a dense quotient of $\C M_1$, which is a closed quotient of some $\C F$.
 So $M$ is a quotient of $\C F$.
\end{proof}
\begin{corollary} %
 A frame is isomorphic to the total part of a strictly zero-dimensional biframe if and only if it is a quotient of the congruence frame of a free frame.
 It is isomorphic to the congruence frame of a frame if and only if this quotient may be taken to be closed.
\end{corollary}
As a corollary, we also obtain a result reminiscent of the classical construction of the Stone-Čech compactification.
\begin{corollary}\label{lem:congruential_coreflection_as_closure} %
 Suppose $L$ is a strictly zero-dimensional biframe and is isomorphic to the biframe $\C F / C$ for some free frame $F$ and some congruence $C$ on $\C F$.
 Then the congruential coreflection of $L$ is given by $\C F / \cl(C)$ together with the obvious map.
\end{corollary}
\begin{proof}
 This follows directly from the argument above and the fact that every quotient may be factored into closed quotient followed by a dense quotient.
\end{proof}

\subsection{Compact strictly zero-dimensional biframes}\label{subsec:cpt_str0d_biframes}

A result of \cite{BanaschewskiFrithGilmour} states that every compact congruence frame is the congruence frame of a Noetherian frame.
In fact, every compact strictly zero-dimensional biframe is of this form.
\begin{lemma}\label{lem:compact_str0d_implies_congruential} %
 Every compact strictly zero-dimensional biframe is the congruence biframe of a Noetherian frame.
\end{lemma}
\begin{proof}
 Let $L$ be a compact strictly zero-dimensional biframe. Every element of $L_1$ is complemented and thus compact in $L_0$.
 Hence $L_1$ is a Noetherian frame. By \cref{thm:dense_quotients_of_congruence_frame},
 the coreflection $\chi: \C L_1 \to L$ is a dense quotient. But a dense frame homomorphism from a regular frame to a compact frame
 is injective and so $\chi$ is an isomorphism.
\end{proof}
\begin{corollary} %
 The functors $\C$ and $\P$ restrict to an equivalence between the category of Noetherian frames and the category of
 compact strictly zero-dimensional biframes.
\end{corollary}
\begin{remark}
 The above equivalence can alternatively be viewed as a restriction of the equivalence between the category of coherent frames and
 the category of compact zero-dimensional biframes, which is itself a restriction of the equivalence described in
 \cite{StablyContinuousFrames} between the categories of stably continuous frames and compact regular biframes.
 We will briefly describe the former equivalence.
\subsubsection*{Aside: Coherent frames and the patch biframe}
\begin{quote}
 The patch biframe $P L$ of a coherent frame $L$ is a sub-biframe of $\C L$ with first part $\nabla L$ and second part
 generated by $\{\Delta_c \mid c \in K(L)\}$. It can be shown to be compact and zero-dimensional. Furthermore, if $f: L \to M$
 is a proper map of coherent frames, $\C f$ restricts to a map $P f: P L \to P M$ turning $P$ into a functor from the category of
 coherent frames and proper frame homomorphisms to the category of compact zero-dimensional biframes. On the other hand, the
 first part of a compact zero-dimensional biframe is coherent and the functors $P$ and $\P$ give an equivalence of categories.
 \par\medskip
 When $L$ is Noetherian, every element is compact and every frame map is proper and so $P$ and $\C$ coincide.
\end{quote}
\smallskip
 The equivalence involving stably continuous frames and the one involving coherent frames are both mentioned in \cite[pp.~19--21]{BiframeThesis}.
 But when restricting further to Noetherian frames, the target category is incorrectly given as the category of compact Boolean biframes,
 instead of the compact strictly zero-dimensional biframes as we have shown above.
 This mistake led to the conclusion that there are infinite compact Boolean biframes, in contrast to the situation
 with frames. While this is not entirely relevant to our current interests, we will take this opportunity to set the record straight.
 The proof uses ideas from theorem 1.14, theorem 4.28 and theorem 6.4 in \cite{LatticesAndOrderedSets}.
\end{remark}
\begin{proposition}\label{thm:compact_Boolean_biframe_finite}
 Under the assumption of the Axiom of Dependent Choice, every compact Boolean biframe is finite.
\end{proposition}
\begin{proof}
 Suppose $M$ is a compact Boolean biframe. Then $M \cong \C L$ for some Noetherian frame $L$ by \cref{lem:compact_str0d_implies_congruential}.
 We show that $L$ is finite.
 
 Assuming dependent choice, Noetherian frames are spatial. (By dependent choice, $L$ is Noetherian if and only if every non-empty subset
 of $L$ has a maximal element. Now suppose $b \not\le a$ and consider the set $S$ of all $c \in L$ such that $a \le c$ and $b \not\le c$.
 This set is non-empty since $a \in S$ and thus $S$ has a maximal element $p$. To show $L$ is spatial, we must show that $p$ is prime.
 Suppose $x,y \not\le p$. Without loss of generality we may assume $x,y > p$. Then $x, y \notin S$ and since $x, y > p \ge a$, we must have
 $x, y \ge b$. But then $x \wedge y \ne p$ and so $p$ is prime.)
 
 Since $\C L$ is Boolean, complementation gives an order-reversing isomorphism between $\nabla L$ and $\Delta L$.
 Similarly to before $\Delta L$ is also a Noetherian frame and thus so is $L\op$. Therefore, $L$ satisfies the descending chain condition.
 
 Notice that since $L$ is spatial, $L$ is finite if and only if $L$ has a finite set of primes. Denote this set by $P$.
 By the above $P$ satisfies both the descending chain condition and the ascending chain condition. So by theorem 1.14 of \cite[p.~17]{LatticesAndOrderedSets}
 (which uses dependent choice), $P$ is finite if it contains no infinite antichains.
 
 Let $A$ be an infinite antichain in $P$. By countable choice, $A$ has an countably infinite subset which we index as $a_0, a_1, a_2, \dots$.
 The sequence $a_0, a_0 \wedge a_1, a_0 \wedge a_1 \wedge a_2, \dots$ must stabilise since $L$ satisfies the descending chain condition.
 Thus there is some $n \in \N$ for which $a_0 \wedge a_1 \wedge \dots \wedge a_n \le a_{n+1}$.
 Now since $a_{n+1}$ is prime, $a_k \le a_{n+1}$ for some $k \le n$. But this contradicts the fact that $A$ is an antichain.
 Thus no infinite antichains can exist.
\end{proof}

\begin{example}\label{ex:lindelof_str0d_biframe}
 We provide an example of a Lindelöf strictly zero-dimensional biframe that is not congruential.
 We will use many results from \cref{section:reflections_of_congruence_frames}, but we present the example here to contrast
 with the \cref{lem:compact_str0d_implies_congruential}.
 Let $L$ be the \emph{chain} of reals in the interval $[0,1]$ with the usual order.
 Since $L$ is hereditarily Lindelöf, $\C L$ is Lindelöf by \cref{lem:lindelof_congruence_frame}.
 Notice that $\D_L = \{(0,0)\} \cup \{(x,y) \mid x, y > 0\} = \langle (x,y) \mid x,y \in (0,1] \cap \Q \rangle$.
 So $\D_L$ is countably generated and thus cozero by \cref{lem:cozero_elements_of_CL}. It is also rare by \cref{lem:rare_congruence}.
 Therefore $\C L / \Delta_{\D_L}$ is a nontrivial dense Lindelöf quotient of $\C L$ by \cref{cor:regular_Lindelof_quotients}. %
\end{example}

\subsection{Strictly zero-dimensional biframes of congruences}\label{subsec:str0d_biframes_of_congruences}

Let $M$ be a strictly zero-dimensional biframe and $\chi: \C M_1 \to M$ its congruential coreflection.
Every element $a \in M_0$ may be associated with a congruence $\chi_*(a)$ on $M_1$, where $\chi_*$ is the
right adjoint of $\chi$. In this way, we may view the elements of any strictly zero-dimensional biframe
as certain congruences on the first part, although the correspondence does not generally preserve joins.

We now take a look at the smallest strictly zero-dimensional biframe over $L$. The congruence $\D_{\C L}$ corresponds
to the nucleus $\chi_*\chi: C \mapsto C^{**}$.
\begin{definition}
 A congruence $C$ on a frame $L$ is called \emph{smooth} if it is a regular element of $\C L$. That is, if $C = C^{**}$ in $\C L$.
\end{definition}
\begin{lemma}
 The image of $\C L / \D$ under $\chi_*$ is the set of smooth congruences on $L$.
\end{lemma}
\begin{remark}\label{rem:biframe_of_smooth_congruences}
 In this way, we may consider $\C L / \D$ itself to be the `biframe of smooth congruences' on $L$.
\end{remark}

The following well-known result allows us to relate the sets of congruences associated with different strictly zero-dimensional biframes.
\begin{lemma}
 Let $\nu$ and $\mu$ be nuclei on a frame $L$ and let $\fix \nu$ denote the set of fixed-points of $\nu$.
 Then $\nu \le \mu$ implies $\fix \mu \subseteq \fix \nu$.
\end{lemma}
\begin{proof}
 Assume $\nu \le \mu$ and suppose $a = \mu(a)$.
 Observe that $a \le \nu(a) \le \mu(a) \le a$ and thus $a = \nu(a)$.
\end{proof}
\begin{corollary}\label{cor:smooth_elements_fixed_by_chi}
 If $M$ is a strictly zero-dimensional biframe and $\chi: \C M_1 \to M$ is its congruential reflection, then the image $\chi_*(M)$
 contains every smooth congruence on $M_1$.
\end{corollary}

One can think of strictly zero-dimensional biframes as frames equipped with a prescribed frame of congruences.
For this interpretation to make sense, we would hope that the extremal epimorphisms from a strictly zero-dimensional biframe $L$
are in one-to-one correspondence with the elements of its total part and that these give quotients of the first part by
congruences in $\fix(\chi_*\chi)$.

\begin{lemma}\label{lem:monos_in_str0dbifrm} %
 A morphism $f: L \to M$ in $\StrZdBiFrm$ is monic if and only if it is dense if and only if $\P f$ is injective.
\end{lemma}
\begin{proof}
 The functor $\P$ reflects monomorphisms since it is faithful and preserves monomorphisms since it is a right adjoint.
 
 Let $\chi: \C L_1 \to L$ be the congruential coreflection of $L$.
 If $f$ is dense, then so is $f\chi$ since $\chi$ is also dense. If $f$ is not dense, then neither is $f\chi$ since $\chi$ is surjective.
 Now by \cref{lem:kernel_on_congruence_frame}, $f\chi$ is dense if and only if $\P(f\chi) = \P f$ is injective.
\end{proof}
\begin{lemma}\label{lem:epis_in_str0dbifrm}
 A morphism $f: L \to M$ in $\StrZdBiFrm$ is an extremal epimorphism if and only if it is a closed quotient.
\end{lemma}
\begin{proof}
 Suppose $f$ is an extremal epimorphism. Factorise $f$ as $st$ where $s$ is dense and $t$ is a closed quotient.
 Then since $f$ is extremal, $s$ is an isomorphism and thus $f$ is a closed quotient.
 
 Conversely suppose $f$ is a closed quotient and $f = mg$ where $m$ is a dense. We may factorise $g$ as $h g'$ where
 $g'$ is a closed surjection and $h$ is dense, so that $f = m'g'$ where $m' = mh$.
 
 Evidently $m'$ is dense and $\ker f \ge \ker g'$. It follows that $\cl(\ker f) = \ker g'$ by \cref{cor:dense_quotients_congruence}.
 But $\ker f$ is closed and so $\ker f = \ker g'$ and $m'$ is an isomorphism. Thus, $m$ is a split epimorphism. But $m$ is already a monomorphism
 and therefore an isomorphism.
\end{proof}
\begin{lemma}\label{lem:closed_quotients_of_str0d_biframe}
 If $M$ is a strictly zero-dimensional biframe over $L$ and $a \in M$ then $M/\nabla_a$ is a strictly zero-dimensional biframe
 over $L / \chi_*(a)$.
\end{lemma}
\begin{proof}
 Suppose $M \cong \C L / D$. Then $M/\nabla_a \cong \C L / (D \vee \nabla_{\chi_*(a)})$ by \cref{lem:closed_quotient_of_quotient}
 and $\cl(D \vee \nabla_{\chi_*(a)}) = \nabla_{\chi_*(a)}$ by \cref{cor:closed_congruences_in_quotient}. %
 Thus, $M/\nabla_a$ is a dense quotient of $\C L / \nabla_{\chi_*(a)} \cong \C (L / \chi_*(a))$.
\end{proof}

There is another different way to view strictly zero-dimensional biframes as frames of congruences, which comes from examining how the
congruence frame of a subframe relates to the congruence frame on the parent frame.
The isomorphism $\C (L/C) \cong \C L / \nabla_C$ gives a simple relationship between the congruence frame on a quotient frame
and the congruence frame on the parent frame. The behaviour of congruence frames on subframes is more nuanced.

Let $\iota: L \hookrightarrow M$ be an inclusion of frames. Then $\C \iota$ is dense by \cref{lem:monos_in_str0dbifrm}, but it needn't be injective.
So the image of $\C \iota$ in $\C M$ is a strictly zero-dimensional biframe over $L$, but it needn't be congruential.
\begin{definition}
 If $\iota: L \hookrightarrow M$ is an inclusion of frames, we call $\image(\C \iota)$ the strictly zero-dimensional biframe
 over $L$ \emph{induced by the inclusion} $\iota: L \hookrightarrow M$.
\end{definition}
\begin{definition}
 We say a biframe $M = (M_0, M_1, M_2)$ is a \emph{sub-biframe} of a biframe $L$ if
 $M_0 \subseteq L_0$, $M_1 \subseteq L_1$ and $M_2 \subseteq L_2$.
\end{definition}
Note that the strictly zero-dimensional sub-biframes of a strictly zero-dimensional biframe $N$ are in one-to-one correspondence with the subframes of
$\P N$. In the above situation, $\image(\C \iota)$ is a strictly zero-dimensional sub-biframe of $\C M$.
We can thus view the elements of this strictly zero-dimensional biframe as congruences (and even compute joins in the usual way) ---
but they are congruences on $M$ instead of $L$. In fact, every strictly zero-dimensional biframe may be expressed in this way.
\begin{lemma}
 Every strictly zero-dimensional biframe is induced by an some inclusion.
\end{lemma}
\begin{proof}
 Suppose $L = (L_0, L_1, L_2)$ is a strictly zero-dimensional biframe. Then inclusion of frames $L_1 \hookrightarrow L_0$ induces
 $L$ as strictly zero-dimensional sub-biframe of $\C L_0$.
\end{proof}

\subsection{Skula biframes}\label{subsec:skula_biframe}

Congruence frames have a spatial analogue, the Skula topology \cite{Skula}, which will be of importance in \cref{subsec:spatial_reflection}.
\begin{definition}
 If $(X, \tau)$ is a topological space, the \emph{Skula modification} of $X$ is the space $(X, \sigma)$
 where $\sigma$ is the topology generated by the sets in $\tau$ and their complements.
\end{definition}
This construction is functorial, since any continuous function between topological spaces is also continuous with respect to their Skula modifications.

\begin{remark}
 Recalling \cref{def:T_D}, we notice that the Skula modification of a $T_D$ space is discrete. In fact, the spaces with discrete Skula modification
 are precisely the $T_D$ spaces.
\end{remark}

The Skula modification of $(X,\tau)$ is associated with a natural bitopological structure $(X, \tau, \upsilon)$, where
$\upsilon$ is the topology on $X$ generated by taking the closed sets in $(X, \tau)$ as basic open sets \cite{StrictlyZeroDimensional}.
Then the Skula topology $\sigma$ is the join of these two topologies on $X$.

Whenever $(X, \tau)$ is $T_0$, the Skula modification is zero-dimensional and $T_0$ and so $(X, \tau, \upsilon)$ is a sober bispace
and may thus be recovered from the biframe $(\sigma, \tau, \upsilon)$, which is strictly zero-dimensional.
\begin{definition}
 The biframe $(\sigma, \tau, \upsilon)$ is called the \emph{Skula biframe} of the space $(X, \tau)$.
 The corresponding functor is denoted by $\Sk: \Top\op \to \StrZdBiFrm$.
\end{definition}
\begin{proposition}\label{prop:T0_spaces_equivalent_to_spatial_str0d_biframes}
 The functor $\Sk: \Top_0\op \to \StrZdBiFrm$ gives a dual equivalence of categories between the $T_0$ topological spaces and the
 spatial strictly zero-dimensional biframes.
\end{proposition}

It it possible to rephrase the above construction frame theoretically. The following lemma is immediate.
\begin{lemma}\label{lem:skula_biframe_induced_by_subframe}
 If $(X, \tau)$ is a topological space, then $\Sk X$ is the strictly zero-dimensional biframe over $\tau$ induced by
 the inclusion $\tau \subseteq 2^X$.
\end{lemma}

We may consider the Skula biframe of a frame $L$ by taking the Skula biframe of its spectrum $\Sigma L$.
The resulting biframe $\Sk\Sigma L$ is then a strictly zero-dimensional biframe over the spatial reflection $\Spat L$.

\begin{lemma}\label{lem:skula_spatial_reflection_of_congruence} %
 If $L$ is a frame, $\Sk\Sigma L$ is the spatial reflection of $\C L$.
 Furthermore, the functors $\Sk\Sigma$ and $\Spat\C$ are naturally isomorphic. %
\end{lemma}
\begin{proof}
 Let $\sigma: \C L \twoheadrightarrow \Spat\C L$ be the spatial reflection of $\C L$.
 Since $\Spat L$ is a quotient of $L$, $\Sk\Sigma L$ is a (spatial) quotient of $\C L$
 and thus factors through $\sigma$ to give $\varphi: \Spat\C L \twoheadrightarrow \Sk\Sigma L$.
 
 Any subframe of a spatial frame is spatial and so the first part of $\Spat\C L$ is spatial.
 But $\Spat L$ is the largest spatial quotient of $L$ and so $\P \varphi$ is an isomorphism
 and $\Spat\C L$ is strictly zero-dimensional over $\Spat L$.
 
 Now since $\Spat\C L$ is spatial, we may embed it in a complete atomic Boolean algebra $B = 2^X$.
 We let $\iota: \Spat\C L \hookrightarrow \C B$. Then $\P \iota: \Spat L \hookrightarrow B$ and $\Spat\C L$
 is the strictly zero-dimensional biframe over $\Spat L$ induced by $\P \iota$.
 But by \cref{lem:skula_biframe_induced_by_subframe} the same is true of $\Sk X$ where $X$ is endowed with the topology
 $\image(\P \iota)$. Finally, $\Sk X \cong \Sk\Sigma L$ since the Skula modification commutes
 with the $T_0$ reflection of spaces. Thus, $\varphi$ is an isomorphism.
 
 We now show naturality. Since $\P$ is faithful, $\varphi_L : \Spat\C L \to \Sk\Sigma L$ is natural if and only
 if $\P \varphi_L$ is natural. Consider the following diagram.
 \begin{center}
   \begin{tikzpicture}[node distance=3.5cm, auto,
    back line/.style={color=gray, text=black},
    cross line/.style={preaction={draw=white, -,line width=6pt}, color=black}]
    
    \node (A) {$L$};
    \node [right of=A] (B) {$\P\Spat\C L$};
    \node [below of=A] (C) {$\Spat L$};
    \node [right of=C] (D) {$\P\Sk\Sigma L$};
    
    \node [right of=A, above of=A, node distance=2cm] (A1) {$M$};
    \node [right of=A1] (B1) {$\P\Spat\C M$};
    \node [below of=A1] (C1) {$\Spat M$};
    \node [right of=C1] (D1) {$\P\Sk\Sigma M$};
    
    \draw[->>] (A1) to node {$\P\sigma_{\C M}$} (B1);
    \draw[back line, ->>] (A1) to node [yshift=9pt, xshift=-2pt] {$\sigma_M$} (C1);
    \draw[->>] (B1) to node {$\P\varphi_M$} (D1);
    \draw[back line, double equal sign distance] (C1) to node {} (D1);
    
    \draw[cross line, ->>] (A) to node [swap, yshift=1pt, xshift=-5pt] {$\P\sigma_{\C L}$} (B);
    \draw[cross line, ->>] (A) to node [swap] {$\sigma_L$} (C);
    \draw[cross line, ->>] (B) to node [swap, yshift=-15pt, xshift=1pt] {$\P\varphi_L$} (D);
    \draw[cross line, double equal sign distance] (C) to node {} (D);
    
    \draw[->] (A) to node {$f$} (A1);
    \draw[->] (B) to node [yshift=-6pt, xshift=2pt] {$\P\Spat\C f$} (B1);
    \draw[back line, ->] (C) to node [yshift=1pt, xshift=9pt] {$\Spat f$} (C1);
    \draw[->] (D) to node [swap, yshift=4pt] {$\P\Sk\Sigma f$} (D1);
   \end{tikzpicture}
 \end{center}
 We must show the right face commutes. The front and back faces commute by the definition of $\varphi$,
 the left and top faces commute by the naturality of $\sigma$ and $\P\sigma \C$ respectively and the
 bottom face commutes since $\Omega\Sigma$ and $\P\Sk\Sigma$ are equal identically.
 So we have
 $\P\Sk\Sigma f \circ \P\varphi_L \circ \P \sigma_{\C L} =
  \P\Sk\Sigma f \circ \mathrm{id}_{\Omega\Sigma L} \circ \sigma_L =
  \mathrm{id}_{\Omega\Sigma M} \circ \Spat f \circ \sigma_L =
  \mathrm{id}_{\Omega\Sigma M} \circ \sigma_M \circ f =
  \P\varphi_M \circ \P\sigma_{\C M} \circ f =
  \P\varphi_M \circ \P\Spat\C f \circ \P\sigma_{\C L}$.
  Thus $\P\Sk\Sigma f \circ \P\varphi_L = \P\varphi_M \circ \P\Spat\C f$ since $\P\sigma_{\C L}$
  is an epimorphism and so $\P \varphi_L$ is natural.
\end{proof}
\begin{remark}
 Take note that the image of a closed congruence $\nabla_a$ under $\sigma$ corresponds to an \emph{open}
 set in the topology $\Sigma L$ under the above correspondence. So while we have been thinking of $\Sk \Sigma L$
 as the frame of Skula-open sets in $\Sigma L$, it might be better to think of it as the frame of \emph{Skula-closed}
 sets in $\Sigma L$ ordered by reverse inclusion. Indeed, the results of \cref{subsec:spatial_reflection} allow us
 to conclude that the Skula-closed sets in $\Sigma L$ are in correspondence with the spatial quotients of $L$.
\end{remark}

\begin{remark}
 We can use the Skula biframe to easily conclude some of the results of \cite{SobrificationRemainder}.
 Note that these results explicitly involve non-sober spaces and sobrification. The Skula biframe allows us
 to distinguish between a non-sober $T_0$ space and its sobrification, which is impossible with the
 usual pointfree approach.
 
 The subspaces of a sober space $X$ which have $X$ as their sobrification are in one-to-one
 correspondence with the dense spatial quotients of $\Sk\, X$ (since these have the same first part).
 Recall that $T_0$ space $Y$ is $T_D$ if and only if $\Sk\, Y$ is a discrete strictly zero-dimensional biframe.
 So by \cref{cor:str0d_biframe_boolean_total_part}, $X$ is the sobrification of a $T_D$ space
 if and only if $\C\Omega X / \D$ is spatial and such a $T_D$ space is unique if it exists.
\end{remark}

We now describe how the Skula biframe relates to the discrete strictly zero-dimensional biframe.
\begin{lemma}\label{lem:skula_versus_smallest_str0d_biframe}
 Let $S = \ker(\C L \twoheadrightarrow \Sk\Sigma L)$. The following equivalences hold.
 \vspace{-1.5ex}
 \begin{itemize}
  \item $S \le \D$ if and only if $L$ is spatial.
  \item $S$ is clear if and only if $\Sigma L$ is $T_D$.
  \item $S = \D$ if and only if $L$ is the frame of opens of a sober $T_D$ space.
 \end{itemize}
\end{lemma}
\begin{proof}
 The first equivalence follows from $\Sk\Sigma L$ being strictly zero-dimensional over $\Spat L$.
 For the second equivalence, note that $\Sigma L$ is $T_D$ if and only if $\Sk\Sigma L$ is discrete
 if and only if $S$ is clear.
 The third equivalence is immediate from the first two.
\end{proof}

\subsection{Which morphisms lift?}\label{subsec:which_morphisms_lift}

The largest strictly zero-dimensional biframe over a frame (the congruence biframe) is functorial. The same cannot be said of the smallest
strictly zero-dimensional biframe (the discrete strictly zero-dimensional biframe).
We will see that only some morphisms $f: L \to M$ lift to morphisms $\overline{f}: \C L / \D \to \C M / \D$.
Classifying these morphisms is an interesting problem.

\begin{lemma}
 A morphism $f: L \to M$ lifts to a morphism $\overline{f}: \C L / \D \to \C M / \D$ if and only if $\ker f$ is smooth and
 the strictly zero-dimensional sub-biframe in $\C M / \D$ generated by $\image f$ is discrete.
\end{lemma}
\begin{proof}
 Suppose $\overline{f}: \C L / \D \to \C M / \D$. Factorise $\overline{f}$ as
 $\C L / \D \overset{q}{\twoheadrightarrow} N \overset{\iota}{\hookrightarrow} \C M / \D$,
 where $q$ is a surjection and $\iota$ is an injection. Since $\C L / \D$ is discrete,
 so is $N$ and $q$ is a closed quotient. Now by \cref{lem:closed_quotients_of_str0d_biframe},
 $\ker(\P \overline{f}) = \ker(\P q)$ is smooth and $N$ is clearly a strictly zero-dimensional sub-biframe of $M$ generated
 by $\image(\P \overline{f}) = \image f$.
 
 Conversely, suppose $f: L \to M$ has smooth kernel and $\image f$ generates a strictly zero-dimensional sub-biframe of $\C M / \D$.
 Factorise $f$ as $L \overset{r}{\twoheadrightarrow} P \overset{j}{\hookrightarrow} M$. Since $\ker r$ is smooth, we may view it as
 an element of $\C L / \D$ by \cref{rem:biframe_of_smooth_congruences} and $(\C L / \D) / \nabla_{\ker r}$ is strictly zero-dimensional
 over $P$ by \cref{lem:closed_quotients_of_str0d_biframe}. Now since $\C L / \D$ is discrete, $(\C L / \D) / \nabla_{\ker r}$ is also
 discrete and isomorphic to $\C P / \D$ by \cref{cor:str0d_biframe_boolean_total_part}. The quotient map
 $\C L / \D \twoheadrightarrow (\C L / \D) / \nabla_{\ker r}$ thus gives a lift for $r$.
 Now by the assumption there is an inclusion of $\C P / \D$ into $\C M / \D$ and this gives a lift for $j$. Composing these gives
 a lift for $f$ as required.
\end{proof}
\begin{example}\label{ex:homomorphism_that_doesnt_lift}
 Let $X$ be the sobrification of $\N$ with the cofinite topology and let $L$ be its frame of opens.
 There is a closed congruence $\partial_{p_n}$ for every closed point $n \in X$. But the non-closed point $\omega \in X$
 corresponds to the rare congruence $\partial_{p_\omega}$. So the quotient maps $L \twoheadrightarrow L/\partial_{p_n} \cong \mathbbm{2}$
 lift to maps $\C L / \D \to \mathbbm{2}$, giving points in $\C L / \D$, but the map $L \to L/\partial_\omega$
 does not lift and $\omega$ is not a point of $\C L / \D$. This is related to the fact that $\omega$ was not
 a point in the original $T_D$ space $\N$.
\end{example}

Also, of interest are the morphisms $f: L \to M$ that will lift to morphisms between \emph{any} strictly zero-dimensional biframes over $L$ and $M$.
We say such morphisms lift \emph{absolutely}. The next lemma gives a characterisation of these morphisms, though it is not completely satisfactory
since we still need to mention the congruence functor.
\begin{lemma} %
 A morphism $f: L \to M$ lifts absolutely if and only if $\ker f$ is smooth and the (frame) image of\, $\C f$ is Boolean.
\end{lemma}
\begin{proof}
 It suffices to check that $f$ lifts to a morphism from $\C L / \D$ to $\C M$.
 So we require $\ker (\C f) \ge \D$. This is equivalent to $\ker (\C f) = \partial_C$ for some smooth congruence $C$ on $L$
 by \cref{lem:join_with_largest_dense_cong}.
 
 Now notice that $\ker (\C f) = \partial_C$ if and only if $\cl(\ker (\C f)) = \nabla_C$ and $\ker(\C f)$ is clear.
 That is, if $\ker f = C$ and $\C L / \ker(\C f) \cong \image(\C f)$ is Boolean.
\end{proof}
\begin{corollary}
 For a morphism $f: L \to M$ to lift absolutely, it is sufficient for $\ker f$ to be smooth and $\image f$ to be scattered.
\end{corollary}
\begin{proof}
 Simply notice that $\image (\C f)$ is a strictly zero-dimensional biframe over $\image f$.
 The result then follows from \cref{cor:scattered_frames_have_unique_str0d_biframe_structure}.
\end{proof}

\subsection{Clear elements}\label{subsec:clear_elements}

Let $M$ be a strictly zero-dimensional biframe over $L$ and let $\chi: \C L \to M$ be its congruential coreflection.
By analogy to congruence frames we call elements of $L$ the \emph{closed elements} of $M$ and we may define a function
$\cl: M \to M$ so that $\cl(a)$ is the largest closed element below $a$. As before this map is monotone, deflationary and idempotent
and preserves finite meets. The next lemma shows that it interacts well with the right adjoint of $\chi$.
\begin{lemma}\label{lem:closure_preserved_by_chi_adjoint}
 If $a \in M$ then $\chi_*(\cl(a)) = \cl(\chi_*(a))$.
\end{lemma}
\begin{proof}
 We certainly have $\cl(a) \le a$. By \cref{cor:smooth_elements_fixed_by_chi}, $\chi_*$ preserves closed elements %
 and so $\chi_*(\cl(a))$ is a closed congruence less than $\chi_*(a)$.
 
 Now suppose $\nabla_c \le \chi_*(a)$. Then $\chi(\nabla_c) \le a$. But $\chi(\nabla_c)$ is closed, so $\chi(\nabla_c) \le \cl(a)$
 and $\nabla_c \le \chi_*(\cl(a))$. Thus, $\chi_*(\cl(a))$ is the largest closed congruence below $\chi_*(a)$.
\end{proof}

We may also generalise the notion of a clear congruence to elements of any strictly zero-dimensional biframe.
\begin{definition}
 An element $a$ of a strictly zero-dimensional biframe $M$ is called \emph{clear} if it is the largest element of $M$
 with closure $\cl(a)$.
\end{definition}

\begin{lemma}\label{lem:clear_element_characterisation}
 An element $a \in M$ is clear if and only if $\chi_*(a)$ is a clear congruence if and only if the first part of $M / \nabla_a$ is Boolean.
\end{lemma}
\begin{proof}
 By \cref{lem:closed_quotients_of_str0d_biframe} and \cref{cor:clear_frame_Boolean}, $\chi_*(a)$ is clear if and only if $M / \nabla_a$ is Boolean.
 
 Now let $c = \cl(a)$ and suppose $\chi_*(a)$ is clear. Then $\chi_*(a) = \partial_c$ by \cref{lem:closure_preserved_by_chi_adjoint}.
 If $\cl(b) = c$ then $\nabla_c \le \chi_*(b) \le \partial_c$ and so $b = \chi\chi_*(b) \le a$ and $a$ is clear.
 
 Conversely, suppose $a$ is clear and consider the strictly zero-dimensional biframe $M/\nabla_a$.
 Identifying elements of $M/\nabla_a$ with elements of $M$ lying above $a$, we easily see that the closure of $b \ge a$ in $M/\nabla_a$ is given by
 $\cl(b) \vee a$ in $M$. Thus, since $a$ is clear in $M$, so is the bottom element of $M/\nabla_a$.
 
 We now show that $N = \P(M/\nabla_a)$ is Boolean. Suppose $d$ is a dense element of $N$. This means that $\cl(d^c) = d^* = 0$.
 But $0$ is clear in $M/\nabla_a$ and thus $d^c = 0$.
 So $d = 1$ and $N$ is Boolean by \cref{lem:no_dense_elements_Boolean}.
\end{proof}
\begin{corollary}\label{lem:clear_element_preserved_by_chi_adjoint}
 If $a \in M$ is clear and $c = \cl(a)$, then $\chi_*(a) = \partial_c$.
\end{corollary}
\begin{proof}
 We have $\cl(\chi_*(a)) = \nabla_c$ by \cref{lem:closure_preserved_by_chi_adjoint} and $\chi_*(a)$ clear by \cref{lem:clear_element_characterisation}.
\end{proof}
\begin{corollary}
 If $a \in M$ is clear, then every element $b \ge a$ is also clear.
\end{corollary}

\begin{example}
 Unlike for congruence frames, clear elements might sometimes fail to exist.
 If $L$ is the chain $[0,1]$, then the strictly zero-dimensional biframe $\C L / \D$ has no nontrivial clear elements
 since every clear congruence on $L$ is rare.
\end{example}

In fact, the existence of all clear elements characterises the congruential strictly zero-dimensional biframes.

\begin{theorem}\label{lem:str0d_biframe_congruential_iff_no_missing_clear_elements}
 A strictly zero-dimensional biframe $M$ is congruential if and only if it has no missing clear elements.
\end{theorem}
\begin{proof}
 If $M$ is congruential then all clear congruences exist.
 
 Now suppose $M$ has all clear congruences and take $A \in \C L$ such that $\chi(A) = 1$. Then $\nabla_a \le A \le \partial_a$ for some $a \in L$
 and so $\chi(\partial_a) = 1$. Let $b \in M$ be the clear element with closure $a$.
 Then $\chi_*(b) = \partial_a$ by \cref{lem:clear_element_preserved_by_chi_adjoint}
 and thus $b = \chi\chi_*(b) = \chi(\partial_a) = 1$. But then $a = \cl(b) = 1$ and so $A = 1$.
 Hence $\chi$ is codense. Since $\C L$ is regular, $\chi$ is therefore injective and $M$ is congruential.
\end{proof}

\section{Reflections and coreflections of congruence frames}\label{section:reflections_of_congruence_frames}

Properties of the congruence frame $\C L$ often give us information about hereditary properties of the frame or $\kappa$-frame $L$ itself.
We will take a look at what information we can obtain from the spatial reflection and the compact regular coreflection of
congruence frames. We also consider the congruence frame on the free frame generated by a $\kappa$-frame.

\subsection{Spatial congruences}\label{subsec:spatial_reflection}

In \cref{ex:smallest_dense_sublocale_of_reals}, we saw that quotients of spatial frames need not be spatial.
In this section we will characterise precisely which quotients of a given $\kappa$-frame are spatial and
describe the $\kappa$-frames for which every quotient is spatial. The characterisation was first proved for frames
in \cite{BooleanAssemblies} by a topological argument and again in \cite{SpatialSublocales} with a more
frame-theoretic approach using nuclei. We will use congruences and extend the result to $\kappa$-frames.

\subsubsection*{Aside: $\kappa$-spaces and spatial $\kappa$-frames}
\begin{quote}
 While topological spaces are ubiquitous throughout mathematics, the generalisation of $\kappa$-spaces is seldom encountered.
 One example is provided by the zero-set spaces used to study realcompactness and Wallman-type compactifications, which are
 precisely the regular $\sigma$-spaces. The adjunction between zero-set spaces and regular $\sigma$-frames is discussed in
 \cite{RegularSigmaFrames}. More general $\sigma$-spaces were first studied in \cite{SigmaSpaces}.
 
 \begin{definition}
  A \emph{$\kappa$-space} is a pair $(X, \tau)$ where $X$ is a set and $\tau$ is a sub-$\kappa$-frame of the power set $2^X$.
  The elements of $\tau$ are known as $\kappa$-open sets.
  
  A \emph{$\kappa$-continuous} map between $\kappa$-spaces $(X,\tau)$ and $(Y,\rho)$ is a function $f:X \to Y$ such that
  $f^{-1}(U) \in \tau$ for all $U \in \rho$.
 \end{definition}
 
 There is a contravariant adjunction between $\kappa$-spaces and $\kappa$-frames, which is directly analogous to the one between
 topological spaces and frames (see \cite{SigmaSpaces} for details).
 
 The functor $\Omega: \kappa\Top\op \to \kappa\Frm$ is given by $\Omega(X,\tau) = \tau$ and
 $(\Omega f)(U) = f^{-1}(U)$.
 The functor $\Sigma: \kappa\Frm \to \kappa\Top\op$ is given by
 $\Sigma L = (\Hom(L, \mathbbm{2}), \{U_a \mid a \in L \})$ and $\Sigma f = \Hom(f, \mathbbm{2})$
 where $U_a = \{f \in \Hom(L, \mathbbm{2}) \mid f(a) = 1\}$.
 
 The functor $\Sigma$ is then left adjoint to $\Omega$. The unit $\sigma: 1_{\kappa\Frm} \to \Omega\Sigma$ and
 the counit $\sob\op: \Sigma\Omega \to 1_{\kappa\Top\op}$ of this adjunction are given by
 $\sigma_L(a) = U_a$ and 
 $\sob_X(x)(U) = \begin{cases}
                  1 & \text{ if } x \in U \\ 
                  0 & \text{ otherwise}
                 \end{cases}.$
 
 The map $\sigma_L$ is surjective.
 
 \bigskip%
 
 A $\kappa$-frame homomorphism $f: L \to \mathbbm{2}$ is called a \emph{point} of $L$.
 Such a function is uniquely determined by the $\kappa$-ideal $f^{-1}(\{0\})$. The points of $L$
 correspond bijectively to the \emph{prime} $\kappa$-ideals, which we will also call points.
 \begin{definition}
  A $\kappa$-ideal $I$ on a $\kappa$-frame is called \emph{prime} if whenever $x \wedge y \in I$,
  either $x \in I$ or $y \in I$.
 \end{definition}
 \begin{lemma}
  A $\kappa$-ideal $I$ on a $\kappa$-frame $L$ is prime if and only if $I$ is a prime element of $\h_\kappa L$.
 \end{lemma}
 \begin{definition}
  A $\kappa$-frame $L$ is called \emph{spatial} if the map $\sigma_L : L \to \Omega\Sigma L$ is an isomorphism.
 \end{definition}
 \begin{proposition}\label{prop:spatial_kappa_frame}
  A $\kappa$-frame is spatial if and only if every principal $\kappa$-ideal is a meet of prime $\kappa$-ideals.
 \end{proposition}
\end{quote}
\smallskip

\begin{definition}
 A congruence $C$ on a $\kappa$-frame $L$ is called \emph{spatial} if the quotient $L/C$ is spatial.
\end{definition}
\begin{remark}
 If $X$ is a topological space, a subspace $A \subseteq X$ induces a spatial congruence as described in
 \cref{subsec:congruences_and_subspaces}. However, it is possible that a congruence $C$ on $L$ is spatial
 without being induced by any subspace of $X$. The congruence $\partial_{p_\omega}$ on $\Omega(\N)$ from \cref{ex:homomorphism_that_doesnt_lift}
 is one such example. However, it is shown in \cite[pp.~99--103]{PicadoPultr} that this can never happen when $X$ is sober.
\end{remark}

\begin{remark}
Recall that an element $a$ in a frame $L$ is called spatial if it is the meet of prime elements.
This should not cause confusion however, since we will shortly show that a congruence is spatial
if and only if it is a spatial element of $\C L$.
\end{remark}

We assume throughout this subsection that $L$ is a $\kappa$-frame.
We begin with a characterisation of the prime elements of $\C L$.
\begin{lemma}\label{lem:prime_congruence}
 An element of $\C L$ is prime if and only if it is of the form $\partial_P$ where $P$ is a prime $\kappa$-ideal.
\end{lemma}
\begin{proof}
 The prime elements of $\C L$ are in one-to-one correspondence with the frame homomorphisms from $\C L$ to $\mathbbm{2}$,
 while universal property of $\C L \cong \h_\kappa\Cvar L$ gives a correspondence between frame homomorphisms from $\C L$ to $\mathbbm{2}$
 and $\kappa$-frame homomorphisms from $L$ to $\mathbbm{2}$.
  \begin{center}
   \begin{tikzpicture}[node distance=2.5cm, auto]
    \node (CL) {$\C L$};
    \node (L) [below of=CL] {$L$};
    \node (M) [right of=L] {$\mathbbm{2}$};
    \draw[->] (L) to node {$\nabla_\bullet$} (CL);
    \draw[->, dashed] (CL) to node {$\overline{f}$} (M);
    \draw[->] (L) to node [swap] {$f$} (M);
   \end{tikzpicture}
 \end{center}
 A prime element $C$ of $\C L$ is then an element such that $\cl(\ker \overline{f}) = \nabla_C$ for some
 $\kappa$-frame map $f: L \to \mathbbm{2}$. By \cref{lem:kernel_on_congruence_frame}, we find that $C = \ker f$.
 The map $f$ corresponds to a prime $\kappa$-ideal $P$ so that $\nabla_P = \cl(\ker f)$. Finally, $\ker f$ is clear
 since $\mathbbm{2}$ is Boolean and so $C = \ker f = \partial_P$.
\end{proof}

\begin{lemma}\label{lem:prime_congruence_elements}
 If $P \in \h_\kappa L$ is prime, then $(x,y) \in \partial_P$ if and only if $x \in P \iff y \in P$.
\end{lemma}
\begin{proof}
 By \cref{cor:clear_cong_characterisation}, $(x,y) \in \partial_P$ if and only if $x \wedge z \in P \iff y \wedge z \in P \,\text{ for all $z \in L$}$.
 For $z \in P$, this condition always holds. But when $z \notin P$, we have $x\wedge z \in P$ if and only if $x \in P$ since $P$ is prime.
 There is always such a $z$ since $P \ne 1$ and we obtain the desired result.
\end{proof}

\begin{lemma}
 $L$ is spatial if and only if the intersection of the primes in $\C L$ is $0$.
\end{lemma}
\begin{proof}
 By \cref{prop:spatial_kappa_frame}, $L$ is spatial if and only if for every $x,y \in L$ with $x < y$, there is a prime $\kappa$-ideal $P$ such that
 $x \in P$, but $y \notin P$. But by \cref{lem:prime_congruence_elements}, $x \in P$ and $y \notin P$ just means that $(x,y) \notin \partial_P$.
 So $L$ is spatial precisely when $(x,y) \in \bigwedge_{P \text{ prime}} \partial_P$ implies $x = y$. The result follows.
\end{proof}

\begin{corollary}\label{cor:spatial_congruence}
 A congruence $C$ is spatial if and only if it is an intersection of primes in $\C L$.
\end{corollary}
\begin{proof}
 We simply use $\C(L/C) \cong \,\uparrow\!C \subseteq \C L$ and apply the previous lemma,
 noting that the prime elements in $\uparrow\!C$ are precisely the prime elements in $\C L$ lying above $C$.
\end{proof}
\begin{corollary}
 The frame $\C L$ is spatial if and only if every quotient of $L$ is spatial.
\end{corollary}

\begin{remark}\label{rem:biframe_of_spatial_congruences}
 In \cref{subsec:skula_biframe} we discussed the Skula biframe of a frame $L$.
 In light of the above and \cref{lem:skula_spatial_reflection_of_congruence}, we may
 conclude that $\Sk\Sigma L$ is the strictly zero-dimensional biframe of \emph{spatial}
 congruences on $L$.
\end{remark}

\begin{theorem}\label{thm:spatial_reflection_of_quotients}
 $\Spat(L/C) \cong L/\sigma(C)$ where $\sigma$ is the spatial reflection of $\C L$.
\end{theorem}
\begin{proof}
 We know that $\Spat(L/C)$ is the largest spatial quotient of $L/C$.
 This corresponds to the smallest spatial congruence lying above $C$ in $\C L$.
 But a congruence is spatial if and only if it is the intersection of prime congruences
 and $\sigma(C) = \bigwedge \{\partial_P \mid \partial_P \ge C, \partial_P \text{ prime}\}$
 is then the smallest such congruence above $C$.
\end{proof}

\subsection{Compactification of congruence frames}\label{subsec:congruence_frame_compactification}

\begin{definition}
 A frame $L$ is called \emph{hereditarily $\kappa$-Lindelöf} if all of its quotients are $\kappa$-Lindelöf.
\end{definition}
\begin{lemma}\label{lem:hereditary_lindelof} %
 A frame $L$ is hereditarily $\kappa$-Lindelöf if and only if every element $a \in L$ is $\kappa$-Lindelöf.
\end{lemma}
\begin{proof}
 If $L$ is hereditarily $\kappa$-Lindelöf, then every open quotient $L/\Delta_a \cong \,\downarrow\!\!a$ is $\kappa$-Lindelöf
 and thus every element $a \in L$ is $\kappa$-Lindelöf.
 
 Conversely, suppose every element of $L$ is $\kappa$-Lindelöf.
 Let $h: L \twoheadrightarrow M$ be a surjective frame homomorphism. For any join $a = \bigvee_{\alpha \in I} x_\alpha$ in $M$,
 we can consider $b = \bigvee_{\alpha \in I} h_*(x_\alpha)$ in $L$. Since $b$ is $\kappa$-Lindelöf, this has a subcover of cardinality
 less than $\kappa$, $b = \bigvee_{\alpha \in J} h_*(x_\alpha)$.
 But $h$ is surjective, so $hh_* = \mathrm{id}_M$ and thus $a = h(b)$ and $\bigvee_{\alpha \in J} x_\alpha$
 is a subcover of $\bigvee_{\alpha \in I} x_\alpha$ of the appropriate cardinality.
\end{proof}

It is shown in \cite{BanaschewskiFrithGilmour} that a congruence frame $\C L$ is compact if and only if
the frame $L$ is Noetherian. Their proof easily generalises to give the following lemma. We use $\C_\kappa L$
to denote the frame of $\kappa$-frame congruences on a \emph{frame} $L$.
\begin{lemma}\label{lem:lindelof_congruence_frame} %
 Let $L$ be a frame. The following are equivalent.
 \vspace{-1.5ex}
 \begin{enumerate}
  \itemsep=0em
  \item $\C L$ is $\kappa$-Lindelöf
  \item L hereditarily $\kappa$-Lindelöf
  \item $\C L = \C_\kappa L$ (i.e.\ every $\kappa$-frame congruence on $L$ is a frame congruence)
 \end{enumerate}
\end{lemma}
\begin{proof}
 $(1 \Rightarrow 2)$ Since $\nabla_a \in \C L$ is complemented, it is a $\kappa$-Lindelöf and, since $\nabla_\bullet$ is injective, so is $a \in L$.
 
 $(2 \Rightarrow 3)$ Since $L$ is hereditarily $\kappa$-Lindelöf, so is $L \times L$. So arbitrary joins in $L \times L$ may be replaced by ones
                     of cardinality less than $\kappa$. Thus any sub-$\kappa$-frame of $L \times L$ is already a subframe and any $\kappa$-frame
                     congruence on $L$ is a frame congruence.
 
 $(3 \Rightarrow 1)$ This follows directly from \cref{cor:congruence_kappa_frame_Lindelof}.
\end{proof}

The above lemma prompts us to consider the relationship between $\C_\kappa L$ and $\C L$ in general.\
If $S \subseteq L\times L$, we will use $\langle S \rangle_\kappa$ to denote the $\kappa$-frame congruence generated by $S$
and $\langle S \rangle_\Frm$ for the frame congruence generated by $S$. We recall that $\C_\kappa L$ has a natural biframe structure,
$(\C_\kappa L, \langle\nabla_a \mid a \in L\rangle, \langle\Delta_a \mid a \in L\rangle)$.
\begin{lemma} %
 The map $h: \C_\kappa L \to \C L$ given by $h(C) = \langle C \rangle_{\Frm}$ is a dense surjective biframe homomorphism
 and thus represents a biframe $\kappa$-Lindelöfication of $\C L$.
\end{lemma}
\begin{proof}
 Notice that if $C = \langle(c_\gamma,c'_\gamma) \mid \gamma \in I\rangle_\kappa$ is a $\kappa$-frame congruence on $L$, then
 $h(C) = \langle(c_\gamma,c'_\gamma) \mid \gamma \in I\rangle_\Frm$ and so $h$ clearly preserves joins.
 
 By \cref{lem:open_and_closed_congruences}, principal closed and open $\kappa$-frame congruences coincide and with closed and open frame congruences.
 Furthermore, since meets of congruences are simply intersections, arbitrary meets of these also coincide and are
 thus preserved by $h$. In particular, $h$ preserves finite meets of principal congruences, but these form a generating
 set of $\C_\kappa L$ and thus $h$ preserves all finite meets by distributivity.
 
 Therefore, $h$ is a frame homomorphism. Furthermore, $h$ is evidently part-preserving and thus a biframe homomorphism.
 Surjectivity is obvious since every (closed/open) frame congruence is also a (principal closed/open) $\kappa$-frame congruence,
 while density is clear since if $(x,y) \in C$ then certainly $(x,y) \in h(C)$.  The map $h: \C_\kappa L \to \C L$
 is then a $\kappa$-Lindelöfication since $\C_\kappa L$ is $\kappa$-Lindelöf and zero-dimensional.
\end{proof}
\begin{remark}
  Note that by \cref{cor:clear_cong_characterisation} a clear $\kappa$-frame congruence $\partial_a$ coming from a principal ideal
  and the frame congruence $\partial_a$ coincide. So by \cref{lem:meet_of_clear}, we might
  also describe $h(C)$ as $\bigwedge \{\partial_a \mid \partial_a \ge C\}$.
\end{remark}

It is natural to ask what kind of $\kappa$-Lindelöfication this might be. In particular, is it universal? We will examine this question further
for $\kappa = \aleph_0$ and $\kappa = \aleph_1$.

\begin{theorem}
 The compactification $h:\Clat L \twoheadrightarrow \C L$ described above is the universal biframe compactification of $\C L$.
\end{theorem}
\begin{proof}
 We consider the biframe strong inclusion induced by $h$ on $\C L$. We have $A \vartriangleleft_i B$ in $\C L$
 if and only if $A' \prec_i B'$ in $\Clat L$ for some $A',B' \in (\Clat L)_i$ with $A = h(A')$, $B = h(B')$.
 If we show $A \prec_i B$ in $\C L$ implies $A \vartriangleleft_i B$, then $h$ is universal.
 
 If $\nabla_a \prec_1 \nabla_b$ then $\Delta_a \vee \nabla_b = 1$ in $\C L$.
 But then $\Delta_a \vee \nabla_b = 1$ in $\Clat L$ by \cref{lem:joins_with_principal}.
 So $\nabla_a \prec_1 \nabla_b$ in $\Clat L$ and clearly $h(\nabla_a) = \nabla_a$ and $h(\nabla_b) = \nabla_b$.
 
 If $A = \bigvee_\alpha \Delta_{a_\alpha}$, $B = \bigvee_\beta \Delta_{b_\beta}$ and $A \prec_2 B$ then there is a $\nabla_c$
 such that $A \wedge \nabla_c = 0$ and $B \vee \nabla_c = 1$ in $\C L$. So $A \le \Delta_c \le B$. 
 
 We now work in $\Clat L$. Let $A' = \bigvee_\alpha \Delta_{a_\alpha}$ and $B' = \bigvee\{\Delta_x \mid \Delta_x \le B\}$.
 We have $A = h(A')$ and $B = h(B')$. Now $A \le \Delta_c$ means $\Delta_{a_\alpha} \le \Delta_c$ for all $\alpha$ and
 so $A' \le \Delta_c$. Also, $\Delta_c \le B$ implies $\Delta_c \le B'$. Thus $A' \prec_2 B'$.
\end{proof}
We obtain a result of \cite{StrictlyZeroDimensional} as a corollary.
\begin{definition}
 A biframe is called \emph{strongly zero-dimensional} if its universal biframe compactification is zero-dimensional.
\end{definition}
\begin{corollary}
 The congruence biframe is strongly zero-dimensional.
\end{corollary}

Zero-dimensional compactifications are best analysed through their complemented elements.
The complemented elements of $\Clat L$ have a particularly nice form.
\begin{lemma}\label{lem:complemented_lattice_congruences}
 The complemented elements in $\Clat L$ are precisely the finitely generated congruences.
\end{lemma}
\begin{proof}
 The open and closed congruences are obviously complemented. Since finitely generated congruences are simply
 finite joins of finite meets of these, they are complemented too. Conversely, suppose $C$ is a complemented
 congruence in $\Clat L$. Then $C$ is compact and thus any generating set for $C$ may be replaced by a finite one.
\end{proof}

\begin{example}
 The compactification $h:\C_{\aleph_0} L \twoheadrightarrow \C L$ need not be a universal \emph{frame} compactification.
 If $h$ were a universal zero-dimensional compactification, then every complemented element of $\C L$ would be an image
 of a complemented element (see \cite{ZeroDimensionalCompactifications}). It therefore suffices to exhibit a complemented element
 of $\C L$, which is not finitely generated. Let $L$ be the chain $\omega + 1$. The congruence
 $\langle (2n,2n+1) \mid n < \omega \rangle$ is complemented, but not finitely generated.
\end{example}

\begin{remark} %
 In \cite{FrithCong} the congruence biframe is introduced in order to describe the frame analogue of the Pervin quasi-uniformity.
 This quasi-uniformity, now known as the \emph{Frith quasi-uniformity}, is totally bounded and has $(\prec_1, \prec_2)$ as its corresponding
 strong inclusion. Thus, the universal biframe compactification $h: \Clat L \to \C L$ constructed above is in fact the \emph{bicompletion}
 of $\C L$ with respect to the Frith quasi-uniformity. %
\end{remark}

We now consider the $\kappa = \aleph_1$ case and the Lindelöfication $h: \Csigma L \to \C L$,
though other choices of $\kappa$ are very similar.
Studying this Lindelöfication involves determining the cozero elements of $\Csigma L$ and $\C L$.
\begin{lemma}\label{lem:cozero_elements_of_CsigmaL} %
 The cozero elements of $\Csigma L$ are precisely the countably generated congruences.
\end{lemma}
\begin{proof}
 It is shown in \cite{PseudoCozeroPartFrame} that the cozero elements of a Lindelöf frame are precisely the Lindelöf elements.
 The proof then proceeds as for \cref{lem:complemented_lattice_congruences}.
\end{proof}
\begin{lemma}\label{lem:cozero_elements_of_CL}
 The $\sigma$-biframe $\Coz \C L$ of biframe cozero elements of $\C L$ is the $\sigma$-biframe of countably generated congruences.
\end{lemma}
\begin{proof}
 Every element in the first part of $\Coz \C L$ is complemented and thus cozero.
 Similarly, every open congruence in $\Delta L$ is cozero, as are countable joins of open congruences.
 
 Now let $A \in \Delta L$ be a cozero element. Then $A = \bigvee_{n \in \N} A_n$ with $A_n \prec\prec_2 A$ for each $n$.
 Now if $A_n \prec_2 A$, there is a $b_n \in L$ such that $A_n \wedge \nabla_{b_n} = 0$ and $A \vee \nabla_{b_n} = 1$.
 Thus $A_n \le \Delta_{b_n} \le A$ and so $A = \bigvee_{n \in \N} \Delta_{b_n}$ is countably generated.
 
 The total part of the cozero $\sigma$-biframe is then the $\sigma$-frame generated by the above cozero elements and thus
 consists of the countably generated congruences.
\end{proof}

\begin{lemma}
 The map $h: \Csigma L \to \C L$ is a universal biframe Lindelöfication if and only if every countably generated
 $\sigma$-frame congruence on $L$ is a frame congruence.
\end{lemma}
\begin{proof}
 The universal biframe Lindelöfication of $\C L$ is given by $\lambda: \h \Coz \C L \to \C L$ with $I \mapsto \bigvee I$
 where $\Coz$ and $\h$ are respectively the biframe cozero and biframe $\sigma$-ideal functors.
 Thus, $h: \Csigma L \to \C L$ is universal if and only if $h$ induces an isomorphism of the biframe cozero elements of $\Csigma L$ and $\C L$,
 which by \cref{lem:cozero_elements_of_CsigmaL,lem:cozero_elements_of_CL} are the countably generated $\sigma$-frame and frame
 congruences respectively.
 
 If every element of countably generated $\sigma$-frame congruence is a frame congruence, then $\Coz h$ acts as the identity map
 and thus $h$ is universal.
 
 Conversely, suppose $h$ is universal and so $\Coz h$ is an isomorphism.
 Notice that $h$ acts as the identity on principal congruences. So letting $C \in \Coz \Csigma L$, we have
 $\langle(a,b)\rangle \le C$ if and only if $\langle(a,b)\rangle \le h(C)$.
 But this means $(a,b) \in C$ if and only if $(a,b) \in h(C)$ and thus $C = h(C)$ and $C$ is a frame congruence.
\end{proof}
\begin{remark}
 I do not have an example of a countably generated $\sigma$-frame congruence on a frame that is not a frame congruence.
 If no such example exists then $h:\Csigma L \to L$ is always universal, just as for the compactification $h_{\aleph_0}: \Clat L \to L$.
 This is an opportunity for further investigation.
\end{remark}

\subsection{Coherent congruences}\label{subsec:lindelof_congruences}

The following result provides a link between the $\kappa$-frame congruences on a $\kappa$-frame $L$
and the frame congruences on the frame freely generated by $L$.
This was observed in \cite{CoherentCongruences} for $\kappa = \aleph_0$, but we show that it holds in general.

\begin{lemma}\label{lem:congruences_on_frames_of_ideals}%
 Let $L$ be a $\kappa$-frame.
 There is an injective frame homomorphism $\iota: \C L \hookrightarrow \C \h_\kappa L$. %
 Furthermore, $\h_\kappa(L/C) \cong \h_\kappa L / \iota(C)$ and the latter quotient map is $\kappa$-proper.
\end{lemma}
\begin{proof}
 We define $\iota$ using the universal properties of $\Cvar L$ and $\h_\kappa$ as in the following diagram.
 \begin{center}
   \begin{tikzpicture}[auto]
    \node (L) {$L$};
    \node (hL) [right=2.5cm of L] {$\h_\kappa L$};
    \node (ChL) [right=2.5cm of hL] {$\C\h_\kappa L$};
    
    \node (CL) [below=1.5cm of L] {$\Cvar L$};
    \node (hCL) [below=1.5cm of CL] {$\h_\kappa\Cvar L$};
    \node (hCL') [right=5.9cm of hCL] {$\C L$};
    
    \draw[->] (L) to node {$\downarrow$} (hL);
    \draw[->] (hL) to node {$\nabla$} (ChL);
    \draw[->] (L) to node [swap] {$\nabla$} (CL);
    \draw[->] (CL) to node [swap] {$\downarrow$} (hCL);
    \draw[double equal sign distance] (hCL) to node {} (hCL');
    
    \draw[->, dashed] (CL) to node {} (ChL);
    \draw[->, dashed] (hCL) to node {} (ChL);
    \draw[->, dashed] (hCL') to node [swap] {$\iota$} (ChL);
   \end{tikzpicture}
 \end{center}
 By commutativity, $\iota(\nabla_a) = \nabla_{\downarrow a}$ and so $\iota(C) = \langle (\downarrow\!a,\downarrow\!b) \mid (a,b) \in C\rangle$.
 
 If $I, J \in \h_\kappa L$, we will write $I \sqsubseteq_C J$ if $(\forall a \in I)(\exists b \in J)\ (a,b) \in C$.
 This is a pre-order which is closed under finite meets, $\kappa$-joins and $\kappa$-directed joins.
 Thus $\widehat{C} = \,\sqsubseteq_C \cap \sqsubseteq_C^\mathrm{op}$ is a congruence on $\h_\kappa L$.
 We show that $\iota(C) = \widehat{C}$.
 
 Certainly, $\iota(C) \le \widehat{C}$. Now let $(I,J) \in \widehat{C}$.
 Then $(\forall a \in I)(\exists b_a \in J)\ (\downarrow\!\!a,\downarrow\!\!b_a) \in \iota(C)$.
 So $I = \bigvee_{a \in I} \downarrow\!\!a \sim_{\iota(C)} \bigvee_{a \in I} \downarrow\!\!b_a := K \le J$. Similarly, we have a $K' \in \h_\kappa L$ such that
 $J \sim_{\iota(C)} K' \le I$. Thus $(I, J) = (I, K) \vee (K',J) \in \iota(C)$ and $\widehat{C} \le \iota(C)$ as required.
 
 Therefore $(\downarrow\!a, \downarrow\!b) \in \iota(C) \iff (a,b) \in C$ and $\iota$ is injective.
 
 Now let $f = \h_\kappa\nu_C$ where $\nu_C: L \twoheadrightarrow L/C$ is a quotient map. Clearly $f$ is surjective, $\kappa$-proper
 and $\ker f \ge \widehat{C}$. We show that $\ker f \le \widehat{C}$.
 
 Suppose $(I, J) \in \ker f$. We may express $J$ as a $\kappa$-directed join $J = \bigvee_\alpha {\downarrow\!b_\alpha}$.
 Then $f(I) = f(J) = \bigvee_\alpha {\downarrow\!\nu_C(b_\alpha)}$, which is also $\kappa$-directed and so
 $f(I) = \bigcup_\alpha {\downarrow \!\nu_C(b_\alpha)}$.
 If $a \in I$, then $\nu_C(a) \in f(I)$ and so $\nu_C(a) \le \nu_C(b_\beta)$ for some $\beta$.
 We find that $b_\beta \wedge a \in J$ and also $(a, b_\beta \wedge a) \in C$.
 Therefore $I \sqsubseteq J$. Similarly, we find $J \sqsubseteq I$ and hence $(I,J) \in \widehat{C}$.
 
 Thus, $\ker f = \widehat{C}$ and $\h_\kappa L / \widehat{C} \cong \h_\kappa(L/C)$.
\end{proof}

\begin{definition}
 Let $L$ be a $\kappa$-coherent frame. We call a congruence $C$ on $L$ a \emph{$\kappa$-coherent} congruence if
 it is generated by pairs of $\kappa$-Lindelöf elements of $L$.
\end{definition}
\begin{remark}
 Notice that a congruence on a $\kappa$-coherent frame $L$ is $\kappa$-coherent if and only if it lies in the subframe of $\C L$ generated
 by all the closed congruences and the open congruences of the form $\Delta_a$ where $a$ is $\kappa$-Lindelöf element of $L$.
\end{remark}

\begin{corollary}\label{cor:coherent_quotients}
 Let $L$ be a $\kappa$-coherent frame. A quotient $L/C$ is $\kappa$-coherent with $\kappa$-proper quotient map if and only if
 the congruence $C$ is $\kappa$-coherent.
\end{corollary}
\begin{proof}
 The backward direction follows directly from the previous lemma. We now show the forward direction.
 
 Suppose $q: L \twoheadrightarrow L/C$ is a $\kappa$-proper map between $\kappa$-coherent frames.
 Let $\Lind L$ denote the $\kappa$-Lindelöf elements of $L$. Then $q(\Lind L) \subseteq \Lind L/C$.
 Now since $L$ is $\kappa$-coherent, $\Lind L$ is closed under finite meets and $\kappa$-joins
 and generates $L$ under arbitrary joins. Similar results hold for $q(\Lind L)$ since $q$ is a surjective frame map.
 But every element of $\Lind L/C$ is then a $\kappa$-join of elements of $q(\Lind L)$ and thus $q(\Lind L) = \Lind L/C$.
 Hence $q$ induces a surjective $\kappa$-frame homomorphism $\tilde{q}: \Lind L \twoheadrightarrow \Lind L/C$
 and $q = \h_\kappa \tilde{q}$\, so that $C = \iota(\ker \tilde{q})$.
\end{proof}

An especially interesting application of the above is to completely regular $\kappa$-Lindelöf frames.
\begin{corollary}\label{cor:regular_Lindelof_quotients} %
 Let $L$ be a completely regular $\kappa$-Lindelöf frame for $\kappa \ge \aleph_1$. The $\kappa$-Lindelöf
 quotients of $L$ are precisely the quotients by $\kappa$-coherent congruences.
\end{corollary}
\begin{proof}
 By \cite{Madden}, every completely regular $\kappa$-Lindelöf (with $\kappa \ge \aleph_1$) is $\kappa$-coherent and
 furthermore, $\kappa$-Lindelöf quotients of these have $\kappa$-proper quotients maps. We may now apply \cref{cor:coherent_quotients}.
\end{proof}

In fact, we can say even more. For simplicity, we now restrict ourselves to the $\kappa = \aleph_1$ case,
but everything can be generalised to any $\kappa \ge \aleph_1$ (and even to $\kappa = \aleph_0$ if we
substitute in zero-dimensionality for complete regularity).
\begin{definition}
 A map $f: L \to M$ of completely regular frames is called \emph{coz-onto} if $\Coz f$ is surjective.
\end{definition}

\begin{theorem}\label{lem:lindelofication_of_quotients} %
 Let $L$ be a completely regular Lindelöf frame and $q: L \twoheadrightarrow L/C$ a coz-onto frame quotient.
 Then the canonical map $g: L / \iota\iota_*(C) \to L/C$ is the universal Lindelöfication of $L/C$.
\end{theorem}
\begin{proof}
 Let $r: L \twoheadrightarrow L/\iota\iota_*(C)$ so that $q = g\circ r$ and $\Coz q = \Coz g \circ \Coz r$.
 
 By \cref{cor:regular_Lindelof_quotients}, $L / \iota\iota_*(C)$ is Lindelöf. Thus $g$ is the universal Lindelöfication if and only if
 $\Coz g$ is an isomorphism.
 
 Notice that $\ker(\Coz q) = \iota_*(\ker q) = \iota_*(C)$ and similarly $\ker(\Coz r) = \iota_*(C)$.
 But $\Coz r$ is surjective since $L / \iota\iota_*(C)$ is Lindelöf as in \cref{cor:regular_Lindelof_quotients}
 and $\Coz q$ is surjective by assumption. So the map $\Coz g$ is an isomorphism as required.
\end{proof}

\begin{remark}
 Note that we may compute $\iota\iota_*(C)$ as $\iota\iota_*(C) = \langle (a,b) \in C \mid a,b \in \Coz L \rangle$.
\end{remark}

\begin{remark}
 The spatial analogue of \cref{cor:regular_Lindelof_quotients,lem:lindelofication_of_quotients} was proven in \cite{ExtensionsOfZeroSets}.
 The notion of a coz-onto quotient corresponds to a z-embedded subspace, Lindelöfication corresponds to realcompactification and
 $\sigma$-coherent congruences to $G_\delta$-closed subsets. The frame of $\sigma$-coherent congruences is analogous to the
 $G_\delta$-modification of a Tychonoff space, which is the topology generated by the set of zero sets in the original space.
\end{remark}

Just as congruence frames provide insight into the $\kappa$-Lindelöf quotients of completely regular frames, so can our knowledge
of $\kappa$-Lindelöf quotients tell us about congruence frames. Consider the following well-known result.
\begin{lemma}\label{lem:meets_of_lindelof_congruences} %
 Let $L$ be a frame and $(C_\alpha)_{\alpha \in I}$ a family of
 congruences on $L$ with $|I| < \kappa$ and such that $L/C_\alpha$ is $\kappa$-Lindelöf for all $\alpha \in I$.
 Then $L /(\bigwedge_{\alpha \in I} C_\alpha)$ is $\kappa$-Lindelöf.
\end{lemma}
\begin{proof}
 It suffices to show this in the case that $\bigwedge_{\alpha \in I} C_\alpha = 0$.
 
 Suppose $\bigvee_{\beta \in J} b_\beta = 1$ in $L$. Let $q_\alpha: L \twoheadrightarrow L/C_\alpha$.
 Then $\bigvee_{\beta \in J} q_\alpha(b_\beta) = 1$. Since $L/C_\alpha$ is $\kappa$-Lindelöf, there is a subset
 $K_\alpha \subseteq J$ with $|K_\alpha| < \kappa$ such that $\bigvee_{\beta \in K_\alpha} q_\alpha(b_\beta) = 1$.
 
 Let $K = \bigcup_{\alpha \in I} K_\alpha$. Then $|K| < \kappa$ since $\kappa$ is a regular cardinal.
 Let $a = \bigvee_{\beta \in K} b_\beta$. Then for all $\alpha \in I$, $q_\alpha(a) = 1$. Thus $(a,1) \in \bigwedge_{\alpha \in I}C_\alpha = 0$
 and so $a = 1$ and $(b_\beta)_{\beta \in K}$ is a subcover of $(b_\beta)_{\beta \in J}$ of the appropriate cardinality.
\end{proof}
\begin{corollary}\label{cor:meets_of_coherent_congs}
 Let $L$ be a completely regular $\kappa$-Lindelöf frame. The set of $\kappa$-coherent congruences on $L$ is closed under $\kappa$-meets in $\C L$.
\end{corollary}

\begin{remark}
 It is interesting to consider if there is a more direct proof that the map $\iota: \C L \to \C\h_\kappa L$ preserves $\kappa$-meets,
 at least when $L$ is a completely regular $\kappa$-frame. I do not know of one as of yet, nor do I know if \cref{cor:meets_of_coherent_congs}
 holds for more general $\kappa$-coherent frames.
\end{remark}

\clearpage
\nocite{CoherentCongruences}
\bibliographystyle{abbrv}
\bibliography{bibliography}

\end{document}